\documentclass[12pt, oneside,reqno]{amsart}
\usepackage[utf8]{inputenc}
\usepackage{graphicx} 
\usepackage{amsmath}
\usepackage{amsthm}
\usepackage{amsfonts}
\usepackage{amssymb}
\usepackage{bbm}
\usepackage{mathtools}
\usepackage{fullpage}
\usepackage{tabularx}
\usepackage{multirow}
\numberwithin{equation}{section}
\usepackage{comment}
\usepackage{enumerate}
\usepackage{appendix}
\usepackage{enumitem}
\usepackage{lipsum}
\usepackage[lmargin=1in,rmargin=1in,tmargin=1in,bmargin=1in]{geometry}
\usepackage[ps,all,arc,rotate]{xy}
\usepackage{graphicx, float, epstopdf}
\usepackage{hyperref, color}
\usepackage{centernot}
\usepackage{fancyhdr}
\usepackage{graphics, setspace}
\usepackage{braket}
\usepackage{tikz}
\usepackage{bbm}

\usepackage{biblatex} 
\allowdisplaybreaks[4]                        
 
\newtheorem{theorem}{Theorem}[section]

\newtheorem{lemma}[theorem]{Lemma}
\newtheorem{proposition}[theorem]{Proposition}

\makeatletter
\newcommand{\doublewidetilde}[1]{{%
  \mathpalette\double@widetilde{#1}%
}}
\newcommand{\double@widetilde}[2]{%
  \sbox\z@{$\m@th#1\widetilde{#2}$}%
  \ht\z@=.9\ht\z@
  \widetilde{\box\z@}%
}
\makeatother

\title{A central limit theorem for partitions involving generalised divisor functions}
\author[1]{Madhuparna Das}
\address{Department of Mathematics, University of Exeter, Exeter, EX4 4QE, United Kingdom}
\email{md679@exeter.ac.uk}
\author[2]{Nicolas Robles}
\address{RAND Corporation, Engineering and Applied Sciences, 1200 S Hayes St, Arlington, VA, 22202, USA}
\email{nrobles@rand.org}
\date{}

\subjclass[2020]{11P82, 05A17, 60F05, 11M41}
\keywords{Central limit theorem, partition function, Dirichlet series with a shifted coefficient, Euler product, saddle-point method, converse mapping}

\begin{document}

\maketitle

\begin{center}
{\it To Bruce Berndt, on the occasion of his 86th birthday}   
\end{center}

\begin{abstract}
We let $p_f(n,k)$ count the $f$-restricted coloured partitions of $n$ with $k$ parts, for a positive integer valued function $f$, given by the bivariate generating series
\begin{align*}
Q_f(z,u) 
\coloneqq 1+\sum_{n=1}^{\infty}\sum_{k=1}^{\infty} p_f(n,k) u^kz^n
=\prod_{k=1}^{\infty}(1+uz^k)^{\Delta_f(k)},   
\end{align*}
where $\Delta_f(n)=|f(n+1)-f(n)|$. A part of size $n$ has $\Delta_f(n)$ distinguishable colours, and each coloured part is available at most once. In this article, we establish a central limit theorem for the number of summands in such partitions when $f(n)=\sigma_r(n)$ denotes the generalised divisor function, defined as $\sigma_r(n)=\sum_{d|n}d^r$ for integer $r\geq 2$. This problem is related to the work of Lipnik, Madritsch, and Tichy on restricted partitions into the indexed values $\lfloor m^{\alpha}\rfloor$, with $0<\alpha<1$. A key element of our proof relies on the analytic behaviour of the Dirichlet series
\begin{align*}
\sum_{n=1}^{\infty}\frac{|\sigma_r(n+1)-\sigma_r(n)|}{n^s},    
\end{align*}
for $\mathrm{Re}(s)>r+1$. We study this problem employing the identity involving the Ramanujan sum. Furthermore, we analyse the Euler product arising in this analysis by adopting the argument of Alkan, Ledoan and Zaharescu. The absolute values are handled using an eventual parity pattern, and the required continuation is obtained from the corrected character decomposition.
\end{abstract}

\section{Introduction}

An arithmetic function $f:\mathbb{N}\to\mathbb{R}_{>0}$ is called multiplicative if $f(mn)=f(m)f(n)$ whenever $(m,n)=1$. For the unrestricted background model, suppose that $\lfloor f(m)\rfloor\ge1$ and that each numerical value of $\lfloor f(m)\rfloor$ occurs at finitely many indices. The partition function $p_f^{\mathrm{unr}}(n)$ of a positive integer $n$ into the function $f$ counts the number of ways to write $n$ as follows
\begin{align}\label{partitiondef}
n=\lfloor{f(n_1)}\rfloor+\lfloor{f(n_2)}\rfloor+\cdots+\lfloor{f(n_r)}\rfloor ,   
\end{align}
where $n_1\leq n_2\leq\cdots\leq n_r$. Different indices are treated as different part types, even when their numerical values coincide. Each indexed part may be used repeatedly. It is a well-known fact that one can write the above expression with the formal generating series 
\begin{align}\label{generatingseries}
{Q_{f}^{\mathrm{unr}}}(z) 
=1+\sum_{n=1}^{\infty}{{p_{f}^{\mathrm{unr}}}(n){z^n}}  
=\prod_{n\in\mathbb{N}^{*}} (1-z^{\lfloor{f(n)}\rfloor})^{-1} .
\end{align}
\noindent
The asymptotic formula for integer partitions, that is, when $f(n)=n$ for all $n\in\mathbb{N}$ in~\eqref{generatingseries}, was first established by Hardy and Ramanujan~\cite{hardyramanujan}. Over the past century, the study of partition asymptotics has attracted the attention of numerous mathematicians; see, for example,~\cite{taylor,madhu,semiprime,roth,vaughan}. A general framework in this context was studied by Debruyne and Tenenbaum~\cite{debten}.

\medskip

Let $p_{f}^{\mathrm{unr}}(n)$ and $p_{f}^{\mathrm{unr}}(n,r)$ denote, respectively, the numbers of unrestricted partitions of $n$ and of such partitions with fixed length $r$, as defined in~\eqref{partitiondef}, where $n_1\leq n_2\leq\ldots\leq n_r$. Li and Chen~\cite[Theorem 1.1]{lichen2} showed
an alternative approach to analysing the generating series~\eqref{generatingseries}, which is through the study of
\begin{align}\label{generatingserieswithgapfunction}
Q_f^{\mathrm{unr}}(z)
=\sum_{n=0}^{\infty}p_f^{\mathrm{unr}}(n)z^n
=\prod_{n=1}^{\infty}(1-z^n)^{-g_f(n)}.
\end{align}
where 
\begin{align}\label{gnsetdefinition}
g_f(n)=\#\{m: m\in\mathbb{N}\;\text{and}\;\lfloor{
f(m)}\rfloor=n\}. 
\end{align}
These identities also hold analytically wherever the products converge absolutely. The level multiplicity $g_f(n)$ is distinct from the absolute-gap function $\Delta_f(n)$ used in Section~\ref{mainresults}.
A related problem was studied by Balasubramanian and Luca~\cite{baluluca}, who considered the function $\mathcal F(n)$ counting unordered factorisations of $n$ into integer factors greater than $1$. Partitions of the form $n=\lfloor{\sqrt{n_1}}\rfloor+\cdots+\lfloor{\sqrt{n_r}}\rfloor$ were subsequently studied by Li and Chen~\cite{lichen}.

\medskip

Our main objective of this paper is to establish a central limit theorem for the random variable $\varpi_{f,n}$, which counts the number of summands in a uniformly chosen $f$-restricted coloured partition of $n$ (as defined in Section~\ref{mainresults}) with respect to the multiplicative function $f$. The distribution of the number of summands in ordinary integer partitions was studied by Erd\H{o}s and Lehner~\cite{erdoslehner}. In this context, it is noteworthy that the weighted moments of integer partitions were studied by Richmond~\cite{richmond1,richmond2}. Let $p_{A}(n,r)$ denote the number of partitions of $n$ into $r$ summands from a given set $A$. The $k$-th weighted moment is then defined by
\begin{align}\label{weightedmoments}
t^k_A(n)
=\sum_{r=1}^{n}r^kp_A(n,r),
\end{align}
for arbitrary fixed $k$ as $n\to\infty$. Furthermore, Richmond remarked that the study of these weighted moments supports the work of Erd\H{o}s and Lehner.

\medskip

In a recent study, Lipnik, Madritsch, and Tichy~\cite{lipniketal} established a central limit theorem for the number of summands in restricted partitions into the indexed values $\lfloor m^{\alpha}\rfloor$, for $0<\alpha<1$, with each index used at most once. For our main result, we have considered $f$ to be the generalised divisor function $\sigma_r(n)$ for integers $r\geq 2$.
Our multiplicities are the absolute gaps of $\sigma_r$, rather than the level multiplicities of the indexed small-power sequence. 

The proof presented in~\cite{lipniketal} used a probabilistic approach to obtain the asymptotics of the moment generating function, and Curtiss' theorem to deduce the central limit theorem. Furthermore, the application of converse mapping to study the analytic behaviour of the associated Dirichlet series makes their proof simple and elegant. In our application, we verify the Mellin estimates uniformly in the positive marking variable and establish the required arithmetic properties of the absolute gaps. General limit theorems for the number of summands, including restricted models, were obtained by Hwang~\cite{hwang} and Madritsch and Wagner~\cite{madwag}; the latter allow repeated numerical part sizes through multiplicities. Our result gives an explicit arithmetic instance of this established framework, with evaluated constants and a direct proof.

\medskip

Studying the limiting distribution of $\varpi_{f,n}$ for $f$ a multiplicative function has a different flavour compared to previous works. Our analysis relies on estimates for sums of absolute divisor gaps, due to the involvement of the gap properties $(|f(n+1)-f(n)|)_{n}$ of the multiplicative function $f$. The even-odd separation used below has precedents in De Koninck~\cite[Section 2]{dekoninck} for $r=2$ and Weingartner~\cite[Proposition 6]{weingartner} for real exponents at least $2$. We give the elementary comparison together with the summatory estimate needed for our proof.


\medskip

Throughout the proof, the integer $r\ge2$ is fixed, and implied constants may depend on $r$. We write $A\asymp B$ if both $A\ll B$ and $B\ll A$ hold for positive quantities. Throughout this paper $f\sim g$ means ${f(x)}/{g(x)}\to 1$ as $x\to\infty$. Additionally, $f=o(g)$ and $f=O(g)$ denote $\left|{f(x)}/{g(x)}\right|\to 0$ and $|f(x)|\leq C|g(x)|$ as $x\to\infty$, respectively, for a suitable constant $C>0$.

\section{The Main Results}\label{mainresults}
Let 
\begin{align}\label{ramanujansum} 
c_m(n) 
=\sum_{\substack{b(\bmod{m})\\(b,m)=1}} e\left(\frac{bn}{m}\right), 
\end{align}
denote the Ramanujan sum, where $e(t)=e^{2\pi i t}$, and the $r$-th divisor function is given by\footnote{In Section~\ref{dirichletsection}, we provide a detailed rationale for selecting this particular function within the context of our analysis.}
\begin{align*} 
\sigma_r(n) 
=\sum_{d|n}d^r. 
\end{align*}
\noindent
Let $f$ be a positive integer valued multiplicative function. Define
\begin{align}\label{gfunctiondefintion}
\Delta_f(n)
\coloneqq {|f(n+1)-f(n)|}.
\end{align}

\noindent
For each $j\ge1$, take $\Delta_f(j)$ distinguishable coloured parts of size $j$. An $f$-restricted coloured partition is a finite unordered selection of these coloured parts, with each coloured part used at most once. Equal numerical sizes with different colours may occur together. Let $p_f(n,k)$ count these partitions of total size $n$ and length $k$. Choosing $h$ colours of size $j$ can be done in $\binom{\Delta_f(j)}h$ ways, so these counts are associated with the bivariate generating series
\begin{align}\label{bivariategeneratingseries} 
Q_f(z,u) 
\coloneqq 1+\sum_{n=1}^{\infty}\sum_{k=1}^{\infty} p_f(n,k) u^kz^n
=\prod_{k=1}^{\infty}(1+uz^k)^{\Delta_f(k)}. 
\end{align}
\noindent
Let $\Pi(n)$ denote the collection of $f$-restricted coloured partitions of $n$, and let $\Pi(n,k)$ be the subset of length $k$. Let $p_f(n)=|\Pi(n)|$ and $p_f(n,k)=|\Pi(n,k)|$ denote the cardinalities of the sets $\Pi(n)$ and $\Pi(n,k)$, respectively. Whenever $p_f(n)>0$, we define the random variable $\varpi_{f,n}$, which counts the number of summands in a uniformly chosen $f$-restricted partition. The probability distribution of $\varpi_{f,n}$ is given by
\begin{align*}
\mathbb{P}(\varpi_{f,n}=k)
=\frac{p_f(n,k)}{p_f(n)}.
\end{align*}

\noindent

For example, $\sigma_2(1)=1$, $\sigma_2(2)=5$, and $\sigma_2(3)=10$,
so there are four colours of size $1$ and five of size $2$. Hence
\begin{align*}
p_{\sigma_2}(2,1)=5,\qquad
p_{\sigma_2}(2,2)=\binom42=6,\qquad
p_{\sigma_2}(2)=11.    
\end{align*}
The six partitions in the second count use two different colours of size $1$.
The term $f$-restricted refers precisely to this coloured model; it does not
mean a restriction to the numerical part values $f(m)$.

We henceforth take $f=\sigma_r$ with a fixed integer $r\ge2$.
Lemma~\ref{gapbounds} shows that $p_f(n)>0$ for every positive $n$ and
that $\Delta_f(j)=O_r(j^r)$, so the product converges locally uniformly
for $|z|<1$ and bounded $u$.

We derive the central limit theorem for the random variable $\varpi_{f,n}$ by examining the bivariate generating function $Q_f(z,u)$, which is defined in~\eqref{bivariategeneratingseries}.

\medskip

In the subsequent lemma, we establish the connection between the generating series $Q_f(z,u)$ and the random variable $\varpi_{f,n}$.

\begin{lemma}\label{generatingfunctionlemma}
Consider the bivariate generating series $Q_f(z,u)$ as defined in~\eqref{bivariategeneratingseries}. The following relation holds
\begin{align*}
Q_f(z,u)
=\prod_{k=1}^{\infty}(1+uz^{k})^{\Delta_f(k)}
=1+\sum_{n=1}^{\infty}p_f(n)\mathbb{E}(u^{\varpi_{f,n}})z^n.
\end{align*}
\end{lemma}

\begin{proof}
The first relation follows directly from~\eqref{bivariategeneratingseries}. Furthermore rearranging the terms in~\eqref{bivariategeneratingseries} yields
\begin{align*}
Q_f(z,u)
=&1+\sum_{n=1}^{\infty}\sum_{k=1}^{\infty}p_f(n,k)u^kz^n\\
=&1+\sum_{n=1}^{\infty}p_f(n)z^n\sum_{k=1}^{\infty}\frac{p_f(n,k)}{p_f(n)}u^k\\
=&1+\sum_{n=1}^{\infty}p_f(n)\mathbb{E}(u^{\varpi_{f,n}})z^n,
\end{align*}
thus concluding the lemma.
\end{proof}

Prior to stating our main result, we note that Lipnik et al.~\cite{lipniketal} considered the multiplicities
$g_{\alpha}(j)=\lceil(j+1)^{1/\alpha}\rceil-\lceil j^{1/\alpha}\rceil$
for $0<\alpha<1$, obtained by counting the indices $m$ with
$\lfloor m^{\alpha}\rfloor=j$. Our weight is
$\Delta_f(j)=|\sigma_r(j+1)-\sigma_r(j)|$.

\begin{theorem}\label{cltforsigma}
Let $f(n)=\sigma_r(n)$ denote the $r$-th divisor function for fixed integer $r\geq 2$. Furthermore, let $\varpi_{f,n}$ be the random variable counting the number of summands in a random $f$-restricted partition of $n$ as defined in \eqref{bivariategeneratingseries}. Then, as $n\to\infty$, $\varpi_{f,n}$ is asymptotically normally distributed with centring parameter $\mu_{f,n}$ and scaling parameter $\nu_{f,n}>0$, i.e.,
\begin{align*}
\mathbb{P}\left(\frac{\varpi_{f,n}-\mu_{f,n}}{\nu_{f,n}}<x\right)=\frac{1}{\sqrt{2\pi}}\int_{-\infty}^{x}e^{-{t^2}/{2}}dt+o(1),    
\end{align*}
uniformly for all $x$ as $n\to\infty$. The saddle parameters $\mu_{f,n}$ and $\nu^2_{f,n}$ are given by
\begin{align}\label{meansigma}
\mu_{f,n}
=\sum_{k=1}^{\infty}\frac{\Delta_f(k)}{e^{\eta k}+1}
\sim C_{\mu}(r)n^{\frac{r+1}{r+2}}
\end{align}
and
\begin{align}\label{variancesigma}
\nu^2_{f,n}
=\sum_{k=1}^{\infty}\frac{\Delta_f(k)e^{\eta k}}{(e^{\eta k}+1)^2}-\frac{\left(\sum_{k=1}^{\infty}\frac{\Delta_f(k)ke^{\eta k}}{(e^{\eta k}+1)^2}\right)^2}{\sum_{k=1}^{\infty}\frac{\Delta_f(k)k^2e^{\eta k}}{(e^{\eta k}+1)^2}} 
\sim C_{\sigma}(r)n^{\frac{r+1}{r+2}},
\end{align}
where the positive constants $C_{\mu}(r)$ and $C_{\sigma}(r)$ are defined in~\eqref{meanconstant} and~\eqref{varianceconstant}, respectively, and $\eta$ is the unique positive solution of
\begin{align*}
n=\sum_{k\geq 1}\frac{k\Delta_f(k)}{e^{\eta k}+1}.   
\end{align*}

\noindent
The actual moments satisfy
$\mathbb E\varpi_{f,n}=\mu_{f,n}+o(\nu_{f,n})$ and
$\operatorname{Var}(\varpi_{f,n})\sim\nu_{f,n}^2$.
Additionally, the tail of the distribution is bounded by
\begin{align*}
\mathbb{P}\left(\frac{\varpi_{f,n}-\mu_{f,n}}{\nu_{f,n}}\geq x\right)\leq 
\begin{cases}
e^{-x^2/2}\left(1+O((\log n)^{-3})\right) & \text{if $0\leq x\leq n^{\frac{r+1}{6(r+2)}}/\log n$},\\
e^{-n^{\frac{r+1}{6(r+2)}}x/(2\log n)}\left(1+O((\log n)^{-3})\right) & \text{if $x\geq n^{\frac{r+1}{6(r+2)}}/\log n$},
\end{cases}
\end{align*}
and an analogous statement holds for $\mathbb{P}(\frac{\varpi_{f,n}-\mu_{f,n}}{\nu_{f,n}}\leq -x)$.
\end{theorem}


\subsection{Initial Setup}\label{initialsetup}

In this section, we provide an outline of the proof of Theorem~\ref{cltforsigma}. We adopt the saddle-point approach from~\cite[\S3]{lipniketal} for the initial construction. Our arithmetic input concerns the absolute gaps $|\sigma_r(n+1)-\sigma_r(n)|$ and is established in Section~\ref{dirichletsection}.

Throughout this subsection, $f=\sigma_r$ for a fixed integer $r\geq2$. Let $0<\delta<1$ and set $\delta\leq u\leq \delta^{-1}$. By ``uniformly in $u$", we mean ``uniformly as $\delta\leq u\leq\delta^{-1}$" throughout this paper.

We begin by observing that, to establish the central limit theorem for the random variable $\varpi_{f,n}$, it is sufficient to show that the normalised moment generating function $\mathbb{E}\left(e^{(\varpi_{f,n}-\mu_{f,n})t/\nu_{f,n}}\right)$ converges to $e^{t^2/2}$ as $n\to\infty$, for each fixed real $t$. Finally, we estimate the tail of the distribution using the Chernoff bound.

Since $\mu_{f,n}/\nu_{f,n}$ is a constant for fixed $n$, we can write
\begin{align*}
\mathbb{E}\left(e^{(\varpi_{f,n}-\mu_{f,n})t/\nu_{f,n}}\right)
=e^{-\mu_{f,n}{t}/\nu_{f,n}}\mathbb{E}\left(e^{\varpi_{f,n}t/\nu_{f,n}}\right).
\end{align*}

\noindent
To achieve our goal we need to obtain the coefficient of $z^n$ in the bivariate generating series $Q_f(z,u)$. By Lemma~\ref{generatingfunctionlemma}, the probability generating function is expressed as
\begin{align*}
\mathbb{E}(u^{\varpi_{f,n}})
=\frac{[z^n]Q_f(z,u)}{[z^n]Q_f(z,1)}.
\end{align*}
\noindent
Applying Cauchy integral formula we have
\begin{align*}
\mathcal{Q}_{f,n}(u)
\coloneqq [z^n]Q_f(z,u)
=\frac{1}{2\pi i}\oint_{|z|=e^{-\tau}}\frac{Q_f(z,u)}{z^{n+1}}dz.
\end{align*}
\noindent
A standard transformation gives, for $\tau>0$,
\begin{align}\label{mainintegral}
\mathcal{Q}_{f,n}(u)
=\frac{e^{n\tau}}{2\pi}\int_{-\pi}^{\pi}\exp\left(in\theta+F(\tau+i\theta,u)\right)d\theta,
\end{align}
where
\begin{align}\label{logbivariategeneratingseries}
F(\gamma,u)
\coloneqq \log Q_f(e^{-\gamma},u)
=\log\bigg(\prod_{k=1}^{\infty}(1+ue^{-\gamma k})^{\Delta_f(k)}\bigg).
\end{align}
For positive real $\gamma$ we take the real logarithm, and near the
positive saddle we use its analytic continuation. In~\eqref{mainintegral},
$\exp F(\gamma,u)$ is understood as the original product
$Q_f(e^{-\gamma},u)$ throughout the contour. Put
$B^2=F_{\gamma\gamma}(\tau,u)>0$.

\noindent
To tackle the integral in~\eqref{mainintegral}, we utilise the saddle-point technique, which is often referred to as the method of steepest descent. It can be viewed as a coarse version of the circle method. Accordingly, the primary application of this method is to derive estimates for integrals of the form
\begin{align*}
\int_{-\pi}^{\pi} e^{\Phi(\tau+i\theta)}d\theta,
\end{align*}
for a suitable function $\Phi$. In our case, $\Phi(\gamma)=n\gamma+F(\gamma,u)$. We will now split the integral into two parts. Choose $\theta_n=\tau^{1+3(r+1)/7}$, then we write
\begin{align}\label{integralsplit}
\int_{-\pi}^{\pi} e^{\Phi(\tau+i\theta)}d\theta
=\int_{|\theta|\leq \theta_n} e^{\Phi(\tau+i\theta)}d\theta
+\int_{\theta_n<|\theta|\leq \pi}e^{\Phi(\tau+i\theta)}d\theta.
\end{align}

\noindent
The first part will compute the contribution near the positive real axis, or, equivalently, around $\theta=0$. This will give us the main term, whereas the contribution from the complementary arc i.e, the estimate arising from $e^{\Phi(\tau+i\theta)-\mathrm{Re}(\Phi(\tau))}$ will give us the error term. Using the Taylor expansion up to the second order, the function $\Phi$ can be expressed as

\begin{align}\label{taylorapproximation}
\Phi(\tau+i\theta)
=\Phi(\tau)+i\theta\Phi'(\tau)-\frac{\theta^2}{2}\Phi''(\tau)+O\bigg({|\theta|^3\sup_{|v|\leq|\theta|}|\Phi'''(\tau+iv)|}\bigg).
\end{align}

\noindent
We now choose $\tau$ such that the first-order derivative $\Phi'(\tau)$ vanishes. Thus $\tau=\tau(n,u)$ is defined by the weighted saddle equation
\begin{align*}
n=-F_\gamma(\tau,u)
=\sum_{k\geq1}\frac{u k\Delta_f(k)}{e^{k\tau}+u}.
\end{align*}
The sum on the right is continuous and strictly decreasing from infinity to zero as $\tau$ increases from zero to infinity, by the growth estimates for $\Delta_f$ proved in Section~3. Hence the positive solution is unique.

Thus, the remaining term of the integral in~\eqref{integralsplit} becomes
\begin{align*}
\int_{|\theta|<\theta_n} e^{\Phi(\tau+i\theta)}d\theta
=e^{\Phi(\tau)}\int_{|\theta|<\theta_n} e^{-\tfrac{\theta^2}{2}\Phi''(\tau)}\bigg(1+O\bigg({\theta_n^3\sup_{|v|\leq\theta_n}|\Phi'''(\tau+iv)|}\bigg)\bigg) d\theta.
\end{align*}
The required smallness of this uniform remainder is verified in Lemma~\ref{maintermforsigmar}.

\noindent
We need to analyse the integrand with the terms involving $\Phi''(\tau)$ and $\Phi'''(\tau)$ in order to show that the aforementioned transformation and estimates are valid. For a detailed explanation, we refer to~\cite[\S3]{lipniketal}. Recall the Mellin transform of a function $h(t)$, defined by
\begin{align*}
h^{*}(s)
\coloneqq\mathcal{M}[h;s]
=\int_{0}^{\infty}h(t)t^{s-1}dt,
\end{align*}
which serves as one of the key ingredients in our analysis. Of particular importance is the rescaling rule (see, for instance,~\cite[Theorem 1]{flajoletetal}), given by whenever the interchange is absolutely convergent,
\begin{align*}
\mathcal{M}\left[\sum_{k}\lambda_kh(\mu_kx);s\right]
=h^{*}(s)\sum_{k}\frac{\lambda_k}{\mu_k^s}.
\end{align*}

\noindent
Thus, the Mellin transform for the series with respect to $\gamma$ yields, for $\mathrm{Re}(s)>r+1$,
\begin{align*}
\mathcal{M}\left[\sum_{k}\Delta_f(k)\log(1+ue^{-\gamma k});s\right]
=\mathcal{D}_f(\Delta;s)\mathcal{I}(s,u),
\end{align*}
where
\begin{align}\label{dirichletserieswithG}
\mathcal{D}_f(\Delta;s)
\coloneqq \sum_{n=1}^{\infty}\frac{\Delta_f(n)}{n^s},
\end{align}
and
\begin{align*}
\mathcal{I}(s{,u})
=\int_{0}^{\infty} \log(1+ue^{-\gamma})\gamma^{s-1}d\gamma,
\end{align*}
represents the Mellin transform of $\gamma\to\log(1+ue^{-\gamma})$.

\medskip

The principal advantage of the Mellin transform lies not in the transform itself, but in the inverse mapping, which allows us to examine the singularities of the transformed function to obtain the asymptotic expansion\footnote{We adopt the term `converse mapping' in Theorem~\ref{cmtheorem} from~\cite{lipniketal}, where it is used in the context of the inverse Mellin transform.}. This step is essential in our case, as we will be working with a Dirichlet series with coefficients involving absolute differences of consecutive values of a multiplicative function.

In the following theorem, $\langle a,b\rangle$ denotes the strip
$a<\mathrm{Re}(s)<b$. The fundamental strip is the maximal strip in which
the defining Mellin integral converges absolutely. The notation
$\asymp$ in~\eqref{cmcondition2} denotes equality of principal parts:
after subtraction of the displayed finite sum, the remainder is
holomorphic in the specified strip.

\begin{theorem}[Converse Mapping~{\cite[{Theorem 4}]{flajoletetal}}]\label{cmtheorem}
Let $f(x)$ be a continuous function on the interval $(0,+\infty)$, and suppose that its Mellin transform $f^{*}(s)$ has a non-empty fundamental strip $\langle\alpha,\beta\rangle$. Assume that $f^{*}(s)$ admits a meromorphic continuation to the larger strip $\langle\gamma,\beta\rangle$, for some $\gamma<\alpha$, with only finitely many poles in that region, and is analytic on the line $\mathrm{Re}(s)=\gamma$. Suppose further that there exists a real number $\eta\in(\alpha,\beta)$ such that
\begin{align}\label{cmcondition1}
f^{*}(s)=O(|s|^{-r})
\end{align}
as $|s|\to\infty$ in $\gamma\leq\mathrm{Re}(s)\leq \eta$, for some $r>1$. If, in addition, $f^*(s)$ admits a singular expansion of the form
\begin{align}\label{cmcondition2}
f^{*}(s)\asymp\sum_{(\xi,k)\in A}\frac{d_{\xi,k}}{(s-\xi)^k} \end{align}
for $s\in\langle\gamma,{\eta}\rangle$, where the finite sum contains the principal parts at every pole in that strip, then the function admits the following asymptotic expansion as $x\to 0^{+}$:
\begin{align*}
f(x)
=\sum_{(\xi,k)\in A} d_{\xi,k}\left(\frac{(-1)^{k-1}}{(k-1)!}x^{-\xi}(\log x)^{k-1}\right)+O(x^{-\gamma}).
\end{align*}
\end{theorem}

\noindent
Therefore, in our case, we must consider the singularities of the Dirichlet series $\mathcal{D}_f(\Delta;s)$ and $\mathcal{I}(s,u)$. Thus, we are left with the expression (initially for $0<u\leq1$ and $\mathrm{Re}(s)>0$)
\begin{align*}
\mathcal{I}(s,u)
&=\int_{0}^{\infty} \log(1+ue^{-\gamma})\gamma^{s-1}d\gamma\\
&=-\int_{0}^{\infty}\sum_{j=1}^{\infty} \frac{(-u)^{j}}{j}e^{-j\gamma}\gamma^{s-1}d\gamma
=-\mathrm{Li}_{s+1}(-u)\Gamma(s),
\end{align*}
where $\mathrm{Li}_{s}(u)$ denotes the polylogarithm, and $\Gamma(s)$ represents the Gamma function.

Termwise integration is justified by
$\sum_{j\geq1}u^j j^{-\mathrm{Re}(s)-1}<\infty$ in the stated range.
For all positive $u$, the integral representation
\begin{align}\label{polylogintegral}
-\Gamma(s)\mathrm{Li}_{s}(-u)
=\int_0^\infty x^{s-1}\frac{u}{e^x+u}\;dx,
\qquad \mathrm{Re}(s)>0,
\end{align}
defines the same branch of the polylogarithm on the negative real axis. Integration by parts gives
\begin{align}\label{mellinkernelidentity}
\mathcal{I}(s,u)
=\frac{1}{s}\int_0^\infty x^s\frac{u}{e^x+u}\;dx
=-\Gamma(s)\mathrm{Li}_{s+1}(-u),
\qquad \mathrm{Re}(s)>0,
\end{align}
uniformly for the stated compact interval of positive $u$. The boundary terms vanish because $\log(1+ue^{-x})$ is bounded near zero and decreases exponentially at infinity.

\medskip

We next verify the vertical decay required in~\eqref{cmcondition1}
directly for the Mellin kernel. Fix $0<a<b$ and write $s=\sigma+it$.
The change of variables $x=e^v$ gives
\begin{align*}
\mathcal{I}(\sigma+it,u)
=\int_{\mathbb R}g_{\sigma,u}(v)e^{itv}\;dv,
\qquad
g_{\sigma,u}(v)=e^{\sigma v}\log(1+ue^{-e^v}).    
\end{align*}
For each nonnegative integer $N$, the derivative
$g_{\sigma,u}^{(N)}(v)$ is $O(e^{av})$ as $v\to-\infty$ and
$O(e^{(b+N)v-e^v})$ as $v\to+\infty$, uniformly for
$a\leq\sigma\leq b$ and $\delta\leq u\leq\delta^{-1}$.
Indeed, all derivatives of $\log(1+ue^{-x})$ are bounded for
$0\leq x\leq1$ and are $O(e^{-x})$ for $x\geq1$; each differentiation
with respect to $v$ introduces at most one further factor $e^v$.
Consequently all these derivatives are integrable and their boundary
values vanish. Repeated integration by parts therefore yields
\begin{align}\label{mellinkerneldecay}
\mathcal{I}(\sigma+it,u)
=O_{a,b,N,\delta}\bigl((1+|t|)^{-N}\bigr)
\qquad(a\leq\sigma\leq b).
\end{align}
The defining integral is holomorphic for $\mathrm{Re}(s)>0$ by local
uniform convergence.

\medskip

Section~\ref{dirichletsection} proves that $\mathcal D_f(\Delta;s)$ continues meromorphically to $\mathrm{Re}(s)>r$, has polynomial growth on closed vertical strips there, and has only a simple pole at $s=r+1$, with residue $R_r=2^{-r}\zeta(r+1)$. Hence the product
$\mathcal D_f(\Delta;s)\mathcal{I}(s,u)$ satisfies the required
vertical decay and has a single pole in this half-plane, since
$\mathcal{I}(r+1,u)>0$. We apply Theorem~\ref{cmtheorem} to this product in Section~\ref{sec4}, giving the
contour calculation explicitly.

\medskip
\noindent
By applying these estimates in the converse mapping and completing the integral estimates in Section~\ref{sec4}, we derive an asymptotic expression for $\mathcal{Q}_{f,n}$, given by
\begin{align}\label{saddlepointtheorem}
\mathcal{Q}_{f,n}(u)
=\frac{e^{n\tau+F(\tau,u)}}{{\sqrt{2\pi}\;B}}(1+{O_r(\tau^{2(r+1)/7})}).
\end{align}
\noindent
Recall the moment-generating function
\begin{align*}
M_n(\theta)=\mathbb{E}(e^{(\varpi_{f,n}-\mu_{f,n})\theta/\nu_{f,n}})
=\exp\left(-\frac{\mu_{f,n}\theta}{\nu_{f,n}}\right)\mathbb{E}(e^{\theta\varpi_{f,n}/\nu_{f,n}})
=\exp\left(-\frac{\mu_{f,n}\theta}{\nu_{f,n}}\right)\frac{\mathcal{Q}_{{f,}n}(e^{{\theta}/\nu_{f,n}})}{\mathcal{Q}_{{f,}n}(1)}.
\end{align*}

\noindent
Next, by applying the Taylor expansions in Section~\ref{sec5}, we obtain, for bounded $\theta$ as $n\to\infty$,
\begin{align*}
M_n(\theta)
=\exp\left(\frac{\theta^2}{2}+O_r\left(n^{-2(r+1)/(7(r+2))}\right)\right),
\end{align*}
thereby establishing the central limit theorem for $\varpi_{f,n}$.

\section{The Dirichlet series}\label{dirichletsection}

In this section, we examine the asymptotic behaviour of the Dirichlet series twisted by the shifted divisor coefficient, with the parity factor needed for the absolute gaps, which plays a crucial role in deriving the main term.

\medskip

Throughout this section, $r\ge2$ is a fixed integer, $f=\sigma_r$, and
$b_j=\Delta_f(j)=|\sigma_r(j+1)-\sigma_r(j)|$. The separation between even and odd arguments when $r=2$ was used by De Koninck~\cite[Section 2]{dekoninck}. The corresponding inequality for all real exponents at least $2$ also appears in Weingartner~\cite[Proposition 6 and its proof]{weingartner}.
We give the argument for the fixed integer $r\ge2$ considered here.

\begin{lemma}\label{gapbounds}
For every odd positive integer $j$,
$\sigma_r(j+1)>\sigma_r(j)$. For every sufficiently large even $j$,
$\sigma_r(j+1)<\sigma_r(j)$. There are constants $c_r,C_r>0$ and an
integer $J_r\ge1$ such that
\begin{equation}\label{eq:gap-size}
 c_rj^r\le b_j\le C_rj^r\qquad(j\ge J_r).
\end{equation}
The upper bound holds for all $j\ge1$ after increasing $C_r$.
Moreover, $p_f(n)>0$ for every positive integer $n$, and
\begin{equation}\label{eq:gap-finite}
b_j=(-1)^{j+1}\bigl(\sigma_r(j+1)-\sigma_r(j)\bigr)+e_r(j),
\end{equation}
where $e_r$ has finite support and is given by
\begin{align*}
e_r(j)=
 \begin{cases}
  2\max\{\sigma_r(j+1)-\sigma_r(j),0\},&j\ \text{even},\\
  0,&j\ \text{odd}.
 \end{cases}    
\end{align*}
\end{lemma}

\begin{proof}
Divisor complementation gives
\begin{align*}
\frac{\sigma_r(j)}{j^r}
=\sum_{d\mid j}d^{-r}.    
\end{align*}
Set
\begin{align*}
a_r=1+2^{-r},\qquad
d_r=(1-2^{-r})\zeta(r)
=1+\sum_{k\ge1}(2k+1)^{-r}.    
\end{align*}
For $k\ge1$ and $r\ge2$,
\begin{align*}
(2k+1)^r=2^r(k+\tfrac12)^r
\ge2^r(k+\tfrac12)^2>2^rk(k+1).    
\end{align*}
It follows that
\begin{align*}
d_r-1<2^{-r}\sum_{k\ge1}\frac1{k(k+1)}=2^{-r}.    
\end{align*}
Thus $\delta_r:=a_r-d_r$ is positive. If $j$ is even, its divisors
$1$ and $2$ give $\sigma_r(j)\ge a_rj^r$. If $j$ is odd, all its
divisors are odd, and hence $\sigma_r(j)\le d_rj^r$.

For odd $j$ we therefore have
\begin{align*}
\sigma_r(j+1)-\sigma_r(j)
\ge a_r(j+1)^r-d_rj^r
\ge\delta_rj^r>0.    
\end{align*}
For even $j$,
\begin{align*}
\sigma_r(j+1)-\sigma_r(j)
\leq j^r\bigl(d_r((1+1/j)^r-1)-\delta_r\bigr).    
\end{align*}
Choose an integer $J_r$ such that
\begin{align*}
J_r\ge
\max\left\{2,
\frac1{(1+\delta_r/(2d_r))^{1/r}-1}\right\}.    
\end{align*}
Then the last upper bound is at most $-\delta_rj^r/2$ for every
even $j\ge J_r$. This proves the asserted signs and the lower bound
in~\eqref{eq:gap-size}, with $c_r=\delta_r/2$.
For all $j\ge1$,
\begin{align*}
b_j\le\max\{\sigma_r(j),\sigma_r(j+1)\}
\le\zeta(r)(j+1)^r\le2^r\zeta(r)j^r,    
\end{align*}
which proves the upper bound. Finally, for odd $j$ the signed difference is positive, while for even $j$ the identity
$|x|=-x+2\max\{x,0\}$ applies. This proves~\eqref{eq:gap-finite};
its correction vanishes for every $j\ge J_r$.

The partition sets are nonempty for every positive integer $n$.
Indeed, every odd size has at least one available colour, and
$b_1=2^r\ge4$. An odd $n$ is represented by a single part of size $n$.
An even $n\ge4$ is represented by parts of sizes $1$ and $n-1$, while
$n=2$ is represented by two different colours of size $1$.
For a fixed total size $n$, only the finitely many colours of sizes
$j\le n$ can occur. Thus each partition set is finite and nonempty,
and $p_f(n)>0$ for all $n\ge1$.
\end{proof}

By substituting the expression for $\Delta_f(n)$ into the Dirichlet series given in~\eqref{dirichletserieswithG}, we arrive at
\begin{align}\label{shiteddirichletseries}
\mathcal{D}_f(\Delta;s)
\coloneqq\sum_{n=1}^{\infty}\frac{{|f(n+1)-f(n)|}}{n^s}
= A_f^{\varepsilon}(s)-D_f^{\varepsilon}(s)+E_f(s),
\end{align}
where $\varepsilon_n=(-1)^{n+1}$,
$E_f(s)=\sum_{n<J_r}e_r(n)n^{-s}$ is an entire Dirichlet polynomial,
and the superscript $\varepsilon$ indicates multiplication of the
$n$-th coefficient by $\varepsilon_n$. We retain the notation
\begin{align}\label{shifteddirichletseries}
A_f(s)
=\sum_{n=1}^{\infty}\frac{f(n+1)}{n^s}
\end{align}
for the Dirichlet series with a shifted coefficient and
\begin{align}\label{dirichletseriesdefinition}
D_{f}(s)
=\sum_{n=1}^{\infty}\frac{f(n)}{n^s},
\end{align}
for $\mathrm{Re}(s)>r+1$.

In particular, the Ramanujan sum $c_m(n)$ possesses arithmetic identities with a broad class of multiplicative functions. Furthermore, it can be interpreted as a finite sum of additive characters, providing a way to deduce the properties of these shifted and twisted series for those multiplicative functions $f:\mathbb{N}\to\mathbb{N}$ that exhibit an identity with the Ramanujan sum.

For the case of generalised divisor function $f(n)=\sigma_r(n)$, we use the identity (see for instance~\cite{kratzel}) for fixed $r\geq 1$,
\begin{align}\label{kthdivisoridentity}
\sigma_r(n)
=\zeta(r+1)n^r\sum_{m=1}^{\infty}\frac{c_m(n)}{m^{r+1}},
\end{align}
with the Ramanujan sum given in~\eqref{ramanujansum}.

For the values $r\ge2$ used here, the required interchanges are absolutely
convergent. Indeed, from
$c_m(n)=\sum_{d\mid(m,n)}d\mu(m/d)$ we have, uniformly in $n$,
\begin{equation}\label{eq:ramanujan-absolute}
\sum_{m\ge1}\frac{|c_m(n)|}{m^{r+1}}
\le \sum_{d\mid n}d\sum_{d\mid m}m^{-r-1}
=\zeta(r+1)\sum_{d\mid n}d^{-r}
\le\zeta(r+1)\zeta(r).
\end{equation}

For Dirichlet character $\chi$ modulo $m$, the generalised Gauss sum is defined as
\begin{align}\label{generalizedgausssum}
c'_{\chi}(n) 
\coloneqq \sum_{\substack{b(\bmod{m})\\(b,m)=1}}\chi(b)e\left(\frac{bn}{m}\right).
\end{align}
Before analysing the Dirichlet series $A_{f}(s)$, we establish some prerequisites.

\begin{lemma}\label{shiftedramanujansumformula}
Let $c_m(n)$ denote the Ramanujan sum for $m\geq 1$ and $n\geq 1$. Then the following identity holds
\begin{align*}
c_m(n+1)
=\frac{\mu(m)}{\varphi(m)}c_m(n)+\frac{1}{\varphi(m)}\sum_{\substack{\chi(\bmod{m})\\ \chi\ne\chi_0}}\tau(\chi)c'_{\bar{\chi}}(n),
\end{align*}
where $\tau(\chi)=c'_{\chi}(1)$ is defined in~\eqref{generalizedgausssum}. 
\end{lemma}

\begin{proof} 
Considering the Fourier expansion of the additive character (see, for instance,~\cite[\S3.4]{iwanieckowlaski}), we write
\begin{align}\label{fourierexpansion} 
e\left(\frac{b}{m}\right)
=\frac{1}{\varphi(m)}\sum_{\substack{\chi(\bmod{m})}}\Bar{\chi}(b)\tau(\chi), \quad \text{if $(b, m) = 1$}, 
\end{align}
where $\tau(\chi)=\sum_{b(\bmod{m})}\chi(b)e(b/m)$ denotes the Gauss sum associated with the Dirichlet character $\chi$. We simplify the shifted Ramanujan sum as follows

\begin{align}\label{generalizedramanujansum}
c_m(n+1)
=\sum_{\substack{b\bmod{m}\\(b,m)=1}}e\left(\frac{b(n+1)}{m}\right)
=&\sum_{\substack{b\bmod{m}\\(b,m)=1}}e\left(\frac{bn}{m}\right)e\left(\frac{b}{m}\right)\nonumber\\
=&\sum_{\substack{b\bmod{m}\\(b,m)=1}}e\left(\frac{bn}{m}\right)\bigg(\frac{1}{\varphi(m)}\sum_{\chi(\bmod{m})}\Bar{\chi}(b)\tau(\chi)\bigg)\nonumber\\
=&\frac{1}{\varphi(m)}\sum_{\chi(\bmod{m})}\tau(\chi)\sum_{\substack{b\bmod{m}\\(b,m)=1}}\Bar{\chi}(b) e\left(\frac{bn}{m}\right),
\end{align}
where we substituted~\eqref{fourierexpansion} into the second line. We shall isolate the terms associated with principal characters as these are considerably simpler. Consequently, from the preceding expression, we derive
\begin{align*}
c_m(n+1)
=&\frac{1}{\varphi(m)}\bigg(\sum_{\substack{\chi(\bmod{m})\\ \chi=\chi_0}}\tau(\chi)\sum_{\substack{b\bmod{m}\\(b,m)=1}}\Bar{\chi}(b) e\left(\frac{bn}{m}\right)+\sum_{\substack{\chi(\bmod{m})\\ \chi\ne\chi_0}}\tau(\chi)\sum_{\substack{b\bmod{m}\\(b,m)=1}}\Bar{\chi}(b) e\left(\frac{bn}{m}\right)\bigg)\nonumber\\
=&\frac{\mu(m)}{\varphi(m)}c_m(n)+\frac{1}{\varphi(m)}\sum_{\substack{\chi(\bmod{m})\\ \chi\ne\chi_0}}\tau(\chi)c'_{\bar{\chi}}(n),
\end{align*}
thereby concluding the lemma.
\end{proof}

In the previous lemma, we transformed the shifted Ramanujan sum into a sum of generalised Gauss sums by using the character sum. This approach facilitates the analysis of the Dirichlet series that arises in the preceding argument.

\begin{lemma}\label{dsklemma}
For a fixed integer $r\geq 1$, consider the double Dirichlet series
\begin{align}\label{dskdefinition}
D_1(s,r)
\coloneqq &\sum_{m=1}^{\infty}\frac{\mu(m)}{\varphi(m)m^{r+1}}\sum_{n=1}^{\infty}\frac{c_m(n)}{n^s},
\end{align}
where $c_m(n)$ denotes the Ramanujan sum. The following identity then holds for $\mathrm{Re}(s)>1$:
\begin{align*}
D_1(s,r)
=\frac{\zeta(s)}{\zeta(s+r+1)}K_r(s),
\end{align*}
where
\begin{align}\label{arsdefinition}
K_r(s)
=C(r)\prod_{p}\left(1+A_{p,r}(s)\right)
\end{align}
with 
\begin{align*}
A_{p,r}(s)
=-\frac{1}{p^{r+1}(p-1)+1},
\frac{p^{-s}(1-p^{-r-1})}{1-p^{-s-r-1}}
\end{align*}
and
\begin{align}\label{crdefinition}
C(r)
\coloneqq\prod_{p}\left(1+\frac{1}{p^{r+1}(p-1)}\right).    
\end{align}
\end{lemma}

\begin{proof}
Since $|c_m(n)|\le\varphi(m)$, the defining double series converges
absolutely for $\mathrm{Re}(s)>1$. Thus the following rearrangements are justified.
Using relation (3.3) of~\cite[\S 3.2]{iwanieckowlaski}, the following expression simplifies as
\begin{align*}
\frac{c_m(n)}{\varphi(m)}
=\frac{1}{\varphi(m)}\frac{\varphi(m)\mu\left({m}/{(m,n)}\right)}{\varphi\left({m}/{(m,n)}\right)}
=\frac{\mu\left({m}/{(m,n)}\right)}{\varphi\left({m}/{(m,n)}\right)}.
\end{align*}

\noindent
By rearranging the terms in~\eqref{dskdefinition} and substituting the above relation, we obtain
\begin{align}\label{dskmodified}
D_1(s,r)
=\sum_{m=1}^{\infty}\frac{\mu(m)}{m^{r+1}}\sum_{n=1}^{\infty}\frac{\mu({m}/{(m,n)})}{\varphi\left({m}/{(m,n)}\right)n^{s}}.
\end{align}

\noindent
Let $\delta=(m,n)$, then we may write $m=k\delta$ and $n=\ell\delta$. Consequently, upon making the change of variables in \eqref{dskmodified}, we obtain
\begin{align*}
D_1(s,r) 
=&\sum_{\delta\geq 1}\mathop{\sum\sum}_{(m,n)=\delta}\frac{\mu(m)}{m^{r+1}}\frac{1}{n^{s}}\frac{\mu({m}/{(m,n)})}{\varphi\left({m}/{(m,n)}\right)}\\
=&\sum_{\delta\geq 1}\mathop{\sum\sum}_{(k,\ell)=1}\frac{\mu(k\delta)}{(k\delta)^{r+1}}\frac{1}{(\ell\delta)^{s}}\frac{\mu(k)}{\varphi(k)}.
\end{align*}
\noindent
Rearranging the order of summation yields
\begin{align*}
D_1(s,r)
=\sum_{k=1}^{\infty}\frac{\mu^2(k)}{\varphi(k)k^{r+1}}{\sum_{\substack{\delta\ge1\\(\delta,k)=1}}}\frac{\mu(\delta)}{\delta^{s+r+1}}\sum_{(\ell,k)=1}\frac{1}{\ell^{s}}.
\end{align*}
\noindent
The sum over $\ell$ equals to
\begin{align*}
\zeta(s)\prod_{p|k}\left(1-\frac{1}{p^{s}}\right),
\end{align*}
and the sum over $\delta$ equals to
\begin{align*}
\frac{1}{\zeta(s+r+1)}\prod_{p|k}\left(1-\frac{1}{p^{s+r+1}}\right)^{-1}.
\end{align*}

\noindent
By rearranging the sum over $k$ and invoking multiplicativity, we deduce that
\begin{align}\label{dskeulerproductform}
D_1(s,r)
=\frac{\zeta(s)}{\zeta(s+r+1)}\prod_p\left(1+\frac{1}{p^{r+1}(p-1)}\frac{1-p^{-{s}}}{1-p^{-(s+r+1)}}\right)
=\frac{\zeta(s)}{\zeta(s+r+1)}K_r(s).
\end{align}

\noindent
Observe that as $\mathrm{Re}(s)\to\infty$, the factors $1-p^{-s}$ and $1-p^{-(s+r+1)}$ tend to 1 for any fixed integer $r\geq 1$. In contrast, the factor 
${1}/{p^{r+1}(p-1)}$ remains unchanged, as it is independent of $s$. Hence, we normalise the Euler product $K_r(s)$ as
\begin{align*}
K_r(s)
=&\prod_p\frac{\left(1+\frac{1}{p^{r+1}(p-1)}\frac{1-p^{-{s}}}{1-p^{-(s+r+1)}}\right)}{1+\frac{1}{p^{r+1}(p-1)}}\left(1+\frac{1}{p^{r+1}(p-1)}\right)\\
=&{C(r)}\prod_{p}\frac{\left(1+\frac{1}{p^{r+1}(p-1)}(1+B_{p,r}(s))\right)}{1+\frac{1}{p^{r+1}(p-1)}},
\end{align*}
where $C(r)$ is given in~\eqref{crdefinition}, and
\begin{align}\label{bprs}
B_{p,r}(s)
=\frac{1-p^{-s}}{1-p^{-s-r-1}}-1
=-\frac{p^{-s}(1-p^{-r-1})}{1-p^{-s-r-1}}.
\end{align}
In this context, it is worth noting that the decomposition argument employed herein was adopted from \cite[\S3]{alexandru}. Thus, the Euler product in~\eqref{dskeulerproductform} can be expressed as
\begin{align*}
K_r(s)=C(r)\prod_{p}(1+A_{p,r}(s)) ,   
\end{align*}
where
\begin{align*}
A_{p,r}(s)
&=\frac{1+\dfrac{1}{p^{r+1}(p-1)}(1+B_{p,r}(s))}
        {1+\dfrac{1}{p^{r+1}(p-1)}}-1\\
&=-\frac{1}{p^{r+1}(p-1)+1}\;
  \frac{p^{-s}(1-p^{-r-1})}{1-p^{-s-r-1}}.
\end{align*}
By~\eqref{bprs}, the above assertion follows, thereby concluding the lemma.
\end{proof}

Let $s=\sigma+it$ with $\sigma, t\in\mathbb{R}$. We have
\begin{align*}
\left|1-\frac{1}{p^{s+r+1}}\right|  
\geq 1-\frac{1}{2^{\sigma+r+1}},\quad\mathrm{and}\quad
\left|\frac{1}{p^s}\right|
= \frac{1}{p^{\sigma}},
\end{align*}
then
\begin{align*}
 |B_{p,r}(s)|
 \le\frac{p^{-\sigma}(1-p^{-r-1})}{1-p^{-\sigma-r-1}}
 <\frac{p^{-\sigma}}{1-2^{-\sigma-r-1}}
 \qquad(\sigma>-r).
\end{align*}
\noindent
Additionally, for a fixed integer $r\geq 1$,
\begin{align*}
{\frac{1}{p^{r+1}(p-1)}}\left({1+\frac{1}{p^{r+1}(p-1)}}\right)^{-1}<\frac{1}{p^{r+1}}.    
\end{align*}
\noindent
so that
\begin{align*}
\sum_{p}|A_{p,r}(s)|
<\frac{2^{\sigma+r+1}}{2^{\sigma+r+1}-1}\sum_{p}\frac{1}{p^{\sigma+r+1}}
<\infty.
\end{align*}
\noindent
It follows that the Euler product $K_r(s)$ converges absolutely (see, for instance,~\cite[\S1.42]{titchmarsh}) and defines an analytic function on the half-plane $\mathrm{Re}(s)>-r$. Furthermore, for any fixed $\sigma_0>-r$,
\begin{align*}
|K_r(s)|
\leq C(r)\prod_{p}\left(1+|A_{p,r}(s)|\right)
\leq C(r)\prod_{p}\left(1+\frac{2^{\sigma_0+r+1}}{(2^{\sigma_0+r+1}-1)p^{\sigma_0+r+1}}\right).
\end{align*}

\noindent
Hence, $K_r(s)$ is uniformly bounded on the half-plane $\mathrm{Re}(s)>\sigma_0$. Since $\zeta(s+r+1)$ has no zeros for
$\mathrm{Re}(s)>-r$, equation~\eqref{dskeulerproductform} gives a
meromorphic continuation of $D_1(s,r)$ to this half-plane. Its only
possible pole there is the simple pole of $\zeta(s)$ at $s=1$.

\medskip

From Lemma~\ref{dsklemma} and the foregoing discussion, we conclude the following result.

\begin{proposition}\label{arsconvergence}
For a fixed integer $r\geq 1$ and $s=\sigma+it$ with $\sigma,t\in\mathbb{R}$, the Dirichlet series $D_1(s,r)$ has a meromorphic continuation to the half-plane $\mathrm{Re}(s)>-r$, where its only pole is a simple pole at $s=1$, with residue $1$.
\end{proposition}

\begin{proof}
The continuation has just been proved. At $s=1$, each factor in the
Euler product for $K_r(s)$ equals
\begin{align*}
1+\frac{1}{p^{r+1}(p-1)}
\frac{1-p^{-1}}{1-p^{-r-2}}
 =(1-p^{-r-2})^{-1}.    
\end{align*}
Hence $K_r(1)=\zeta(r+2)$, and
$\operatorname*{Res}_{s=1}D_1(s,r)=K_r(1)/\zeta(r+2)=1$.
\end{proof}

To analyse the non-principal character contribution, write
\begin{align}\label{eq:normalized-coefficients}
a_n&=\sum_{m\ge1}\frac{\mu(m)c_m(n)}{\varphi(m)m^{r+1}},
&g_n&=\sum_{m\ge1}\frac{c_m(n)}{m^{r+1}}
=\frac{\sigma_r(n)}{\zeta(r+1)n^r},\notag\\
h_n&=g_{n+1}, &d_n&=h_n-a_n.
\end{align}
Thus $D_1(s,r)=\sum_{n\ge1}a_nn^{-s}$. We also put
\begin{align*}
D_1^{\varepsilon}(s,r)=\sum_{n\ge1}\frac{\varepsilon_na_n}{n^s},\qquad
D_2^{\varepsilon}(s,r)=\sum_{n\ge1}\frac{\varepsilon_nd_n}{n^s},
\qquad q_r=2^{-r-1}.    
\end{align*}
These series initially converge absolutely for $\mathrm{Re}(s)>1$.

\begin{lemma}\label{gausssumlemma}
Let $r>1$ be fixed, and define
\begin{align}\label{d2sksum}
D_2(s,r)
\coloneqq\sum_{m=1}^{\infty}\frac{1}{\varphi(m)m^{r+1}}\sum_{\substack{\chi(\bmod{m})\\\chi\ne\chi_0}}\tau(\chi)\sum_{n=1}^{\infty}\frac{c'_{\Bar{\chi}}(n)}{n^{s}}.
\end{align}
\noindent
Let $\mathrm{Re}(s)=\sigma$, then the following inequality
\begin{align*}
|D_2(s,r)|\leq \zeta(\sigma)\zeta(r)
\end{align*}
holds for $\sigma>1$. Moreover,
$D_2(s,r)=\sum_{n\ge1}d_nn^{-s}$, and both $D_2(s,r)$ and
$D_2^{\varepsilon}(s,r)$ continue holomorphically to
$\mathrm{Re}(s)>0$. On each half-plane $\mathrm{Re}(s)\ge\sigma_0>0$,
they are $O_{r,\sigma_0}(1+|s|)$.
\end{lemma}

\begin{proof}
Character orthogonality, applied to the two sums in
\eqref{generalizedgausssum}, gives
\begin{align*}
\sum_{\chi\;({\rm mod}\;m)}|\tau(\chi)|^2=\varphi(m)^2, \qquad
\sum_{\chi\;({\rm mod}\;m)}|c'_{\bar\chi}(n)|^2=\varphi(m)^2.    
\end{align*}
Indeed, after expanding either square, the character sum is
$\varphi(m)$ when the two reduced residues agree and zero otherwise;
the remaining exponential factor then equals $1$.
Consequently, Cauchy's inequality yields, uniformly in $n$,
\begin{align*}
\frac{1}{\varphi(m)}\sum_{\substack{\chi\;({\rm mod}\;m)\\\chi\ne\chi_0}}|\tau(\chi)c'_{\bar\chi}(n)|\le\varphi(m)\le m.    
\end{align*}
The full series of absolute values in~\eqref{d2sksum} is therefore at most
\begin{align*}
\sum_{n\ge1}n^{-\sigma}\sum_{m\ge1}m^{-r}
=\zeta(\sigma)\zeta(r).    
\end{align*}
This proves the stated bound and permits rearrangement for $\sigma>1$.
Lemma~\ref{shiftedramanujansumformula} then gives
\begin{align*}
\sum_{m\ge1}\frac{1}{\varphi(m)m^{r+1}}
\sum_{\substack{\chi\;({\rm mod}\;m)\\\chi\ne\chi_0}}
\tau(\chi)c'_{\bar\chi}(n)
=h_n-a_n=d_n.    
\end{align*}
We next establish the bounded partial sums needed for continuation.
For $x\ge1$, the divisor formula for $c_m(n)$ gives
\begin{align}\label{eq:ramanujan-means}
\sum_{n\le x}c_m(n)
&=x\;\boldsymbol{1}_{m=1}+O(\sigma_1(m)),\notag\\
\sum_{n\le x}\varepsilon_nc_m(n)
&=-x\;\boldsymbol{1}_{m=2}+O(\sigma_1(m)).
\end{align}
The first formula follows by replacing $\lfloor x/d\rfloor$ by
$x/d+O(1)$ in $\sum_{d\mid m}d\mu(m/d)\lfloor x/d\rfloor$. For the second formula, if $d$ is odd then $\sum_{\ell\le x/d}\varepsilon_{d\ell}=O(1)$; if $d$ is even then this sum is $-x/d+O(1)$. Its main coefficient is thus $-\sum_{d\mid m,\;2\mid d}\mu(m/d)$. For odd $m$ this sum vanishes. Writing $m=2^a v$ with $v$ odd shows that it also vanishes when $v>1$ or $a\ge2$, and equals $1$ when $m=2$. This proves~\eqref{eq:ramanujan-means}.
Since $\varepsilon_{n-1}=-\varepsilon_n$, shifting $n$ by $1$ gives
\begin{align}\label{eq:shifted-ramanujan-means}
\sum_{n\le x}c_m(n+1)
&=x\;\boldsymbol{1}_{m=1}+O(\sigma_1(m)),\notag\\
\sum_{n\le x}\varepsilon_nc_m(n+1)
&=x\;\boldsymbol{1}_{m=2}+O(\sigma_1(m)).
\end{align}
The omitted endpoint terms are bounded by $\sigma_1(m)$, so the same
error is valid uniformly in $m$ and $x$. All these errors can be summed over $m$, because
\begin{align*}
\sum_{m\ge1}\frac{\sigma_1(m)}{m^{r+1}}
=\zeta(r+1)\zeta(r)<\infty.    
\end{align*}
In the sums for $a_n$, the extra factor $|\mu(m)|/\varphi(m)$ is at
most $1$. Since $\mu(2)=-1$ and $\varphi(2)=1$, the preceding formulas
give
\begin{align}\label{eq:coefficient-means}
\sum_{n\le x}a_n&=x+O_r(1),
&\sum_{n\le x}h_n&=x+O_r(1),\notag\\
\sum_{n\le x}\varepsilon_na_n&=q_rx+O_r(1),
&\sum_{n\le x}\varepsilon_nh_n&=q_rx+O_r(1),\notag\\
\sum_{n\le x}\varepsilon_ng_n&=-q_rx+O_r(1).
\end{align}
It follows that the partial sums of $d_n$ and of $\varepsilon_nd_n$
are both $O_r(1)$. Partial summation now gives
\begin{align*}
D_2(s,r)&=s\int_1^\infty
\bigg(\sum_{n\le x}d_n\bigg)x^{-s-1}\;dx,\\
D_2^{\varepsilon}(s,r)&=s\int_1^\infty
\bigg(\sum_{n\le x}\varepsilon_nd_n\bigg)x^{-s-1}\;dx.
\end{align*}
The integrals converge locally uniformly for $\mathrm{Re}(s)>0$,
and are bounded by $C_r|s|/\mathrm{Re}(s)$. This proves the continuation
and growth assertions.
\end{proof}

\subsection{Parity twists and the shifted divisor series}\label{section31}

The factor at the prime $2$ in Lemma~\ref{dsklemma} identifies the
residue needed for the absolute gaps. More precisely, for
$\mathrm{Re}(s)>1$,
\begin{equation}\label{eq:d1-parity-factor}
 D_1^{\varepsilon}(s,r)
 =\frac{1+q_r-2^{1-s}}{1+q_r-q_r2^{1-s}}D_1(s,r).
\end{equation}
To see this directly, multiplicativity of the Ramanujan sum in its
modulus gives
\begin{align*}
a_n=\prod_p\left(1-\frac{c_p(n)}{(p-1)p^{r+1}}\right).    
\end{align*}
The factor for $p=2$ is $1+q_r$ when $n$ is odd and $1-q_r$ when
$n$ is even. Thus, with $z=2^{-s}$, the ordinary and alternating
local sums are respectively
\begin{align*}
(1+q_r)+(1-q_r)\frac{z}{1-z},\qquad
(1+q_r)-(1-q_r)\frac{z}{1-z}.    
\end{align*}
Their quotient is the factor in~\eqref{eq:d1-parity-factor}; the odd
prime factors agree. Its denominator has no zeros for
$\mathrm{Re}(s)>0$, since
$|q_r2^{1-s}|<2q_r<1+q_r$.
Proposition~\ref{arsconvergence} therefore continues
$D_1^{\varepsilon}(s,r)$ meromorphically to $\mathrm{Re}(s)>0$,
where its only pole is at $s=1$, with residue $q_r$.

\begin{lemma}\label{shfiteddirichletseries}
For a fixed integer $r>1$, the Dirichlet series
\begin{align*}
A_{\sigma_r}(s)
=\sum_{n=1}^{\infty}\frac{\sigma_r(n+1)}{n^s}
=&\zeta(r+1)\sum_{i=0}^{r}\binom{r}{i}D_1(s-i,r)+\zeta(r+1)\sum_{i=0}^{r}\binom{r}{i}D_2(s-i,r),
\end{align*}
converges absolutely for $\mathrm{Re}(s)>r+1$, where $D_1(s,r)$ and $D_2(s,r)$ are defined in~\eqref{dskdefinition} and~\eqref{d2sksum}, respectively.
The parity-twisted series satisfies, in the same half-plane,
\begin{equation}\label{eq:shifted-parity-series}
A_{\sigma_r}^{\varepsilon}(s)
=\zeta(r+1)\sum_{i=0}^{r}\binom{r}{i}
\bigl(D_1^{\varepsilon}(s-i,r)+D_2^{\varepsilon}(s-i,r)\bigr).
\end{equation}
\end{lemma}

\begin{proof}
Employing the identity~\eqref{kthdivisoridentity}, we express
\begin{align*}
A_{\sigma_r}(s)
=\sum_{n=1}^{\infty}\frac{\sigma_r(n+1)}{n^s}
=\zeta(r+1)\sum_{n=1}^{\infty}\frac{(n+1)^{r}}{n^s}\sum_{m=1}^{\infty}\frac{c_m(n+1)}{m^{r+1}}.
\end{align*}
\noindent
Given that $r>1$ is fixed, and under the condition $\mathrm{Re}(s)>r+1$, one may interchange the order of summation by~\eqref{eq:ramanujan-absolute}, thereby obtaining
\begin{align}\label{sum1}
\zeta(r+1)\sum_{m=1}^{\infty}\frac{1}{m^{r+1}}\sum_{n=1}^{\infty}\frac{c_m(n+1)}{n^s}(n+1)^r
=\zeta(r+1)\sum_{i=0}^{r}\binom{r}{i}\sum_{m=1}^{\infty}\frac{1}{m^{r+1}}\sum_{n=1}^{\infty}\frac{c_m(n+1)}{n^{s-i}}.
\end{align}

Now applying Lemma~\ref{shiftedramanujansumformula} yields,
\begin{align*}
&\sum_{m=1}^{\infty}\frac{1}{m^{r+1}}\sum_{n=1}^{\infty}\frac{c_m(n+1)}{n^{s-i}}\\
&=\sum_{m=1}^{\infty}\frac{1}{\varphi(m)m^{r+1}}\sum_{n=1}^{\infty}\frac{1}{n^{s-i}}\bigg(\mu(m)c_m(n)+\sum_{\substack{\chi(\bmod{m})\\ \chi\ne \chi_0}}{\tau(\chi)}c'_{\bar{\chi}}(n)\bigg)\\
&=\sum_{m=1}^{\infty}\frac{\mu(m)}{\varphi(m)m^{r+1}}\sum_{n=1}^{\infty}\frac{c_m(n)}{n^{s-i}}+\sum_{m=1}^{\infty}\frac{1}{\varphi(m)m^{r+1}}\sum_{\substack{\chi(\bmod{m})\\ \chi\ne \chi_0}}\tau(\chi)\sum_{n=1}^{\infty}\frac{c'_{\bar{\chi}}(n)}{n^{s-i}}\\
&=D_1(s-i,r)+D_2(s-i,r).
\end{align*}

\noindent
By applying Lemma~\ref{dsklemma} and Lemma~\ref{gausssumlemma}, in conjunction with~\eqref{sum1}, we thereby obtain the first identity. Multiplying the $n$-th summands by $\varepsilon_n$ in the same absolutely convergent calculation proves~\eqref{eq:shifted-parity-series}.
\end{proof}

We also recall that (see, for instance,~\cite[\S1.3]{titchmarsh}), for ${\mathrm{Re}(s)>r+1}$,
\begin{align}\label{dsigmars}
D_{\sigma_r}(s)
=\sum_{n=1}^{\infty}\frac{\sigma_r({n})}{n^s}
=\zeta(s)\zeta(s-r),
\end{align}
for fixed $r$.

Removing the Euler factor at $2$ gives
\begin{align*}
\sum_{\substack{n\ge1\\n\ {\rm odd}}}\frac{\sigma_r(n)}{n^s}
=(1-2^{-s})(1-2^{r-s})\zeta(s)\zeta(s-r).    
\end{align*}
Consequently,
\begin{equation}\label{eq:unshifted-parity-series}
D_{\sigma_r}^{\varepsilon}(s)
=\bigl(2(1-2^{-s})(1-2^{r-s})-1\bigr)\zeta(s)\zeta(s-r).
\end{equation}
At $s=r+1$, the factor in parentheses equals $-q_r$.
Together,~\eqref{eq:d1-parity-factor}, \eqref{eq:shifted-parity-series}, and~\eqref{eq:unshifted-parity-series} determine the leading pole of~\eqref{shiteddirichletseries}. We also require the following summatory estimate, including its error term.

\begin{proposition}\label{summatorygap}
For real $x\ge1$, put $B_r(x)=\sum_{j\le x}b_j$. Then
\begin{equation}\label{eq:gap-sum}
 B_r(x)=\frac{R_r}{r+1}x^{r+1}+O_r(x^r),
 \qquad R_r=2^{-r}\zeta(r+1).
\end{equation}
\end{proposition}

\begin{proof}
We use the coefficient means in~\eqref{eq:coefficient-means}.
If $\sum_{n\le x}v_n=cx+O(1)$, then for a fixed nonnegative integer $i$
partial summation gives
\begin{equation}\label{eq:weighted-coefficient-mean}
\sum_{n\le x}n^iv_n
=\frac{c}{i+1}x^{i+1}+O_i(1+x^i).
\end{equation}
Here the implied constant also depends on $c$ and the bound in the
assumed $O(1)$ estimate. For $i=0$, this is the assumed estimate. For $i\ge1$, the formula follows
by substituting $cx+O(1)$ into
\begin{align*}
x^i\sum_{n\le x}v_n
-i\int_1^x t^{i-1}\sum_{n\le t}v_n\;dt.    
\end{align*}
The error is $O_i(x^i)$, and the constant lower-endpoint term is absorbed in the same bound.

By the Ramanujan expansion and~\eqref{eq:normalized-coefficients},
\begin{align*}
\sigma_r(n+1)
=\zeta(r+1)(n+1)^r(a_n+d_n).    
\end{align*}
The partial sums of $\varepsilon_na_n$ are $q_rx+O_r(1)$ and those of
$\varepsilon_nd_n$ are $O_r(1)$ by the proof of
Lemma~\ref{gausssumlemma}. Expanding $(n+1)^r$ and using
\eqref{eq:weighted-coefficient-mean}, we obtain
\begin{align*}
\sum_{n\le x}\varepsilon_n\sigma_r(n+1)
&=\zeta(r+1)\sum_{i=0}^r\binom ri
\sum_{n\le x}n^i\varepsilon_n(a_n+d_n)\\
&=\frac{q_r\zeta(r+1)}{r+1}x^{r+1}+O_r(x^r).
\end{align*}
The lower powers $i<r$ contribute $O_r(x^r)$. Similarly, since
$\sigma_r(n)=\zeta(r+1)n^rg_n$ and
$\sum_{n\le x}\varepsilon_ng_n=-q_rx+O_r(1)$, we have
\begin{align*}
\sum_{n\le x}\varepsilon_n\sigma_r(n)
=-\frac{q_r\zeta(r+1)}{r+1}x^{r+1}+O_r(x^r).    
\end{align*}
Subtract these estimates and apply~\eqref{eq:gap-finite}.
The sum of the finitely supported correction is $O_r(1)$, and
$2q_r\zeta(r+1)=R_r$. This proves~\eqref{eq:gap-sum}.
\end{proof}

\begin{proposition}\label{gapdirichlet}
The Dirichlet series $\mathcal D_f(\Delta;s)$ has abscissa of absolute
convergence $r+1$. It continues meromorphically to $\mathrm{Re}(s)>r$,
where its only pole is a simple pole at $s=r+1$, of residue $R_r$.
On every fixed closed vertical strip contained in $\mathrm{Re}(s)>r$,
it has polynomial growth away from that pole.
\end{proposition}

\begin{proof}
The assertion about absolute convergence follows from
Lemma~\ref{gapbounds}. Formula~\eqref{shiteddirichletseries}, combined
with~\eqref{eq:shifted-parity-series} and
Lemma~\ref{gausssumlemma}, gives its meromorphic continuation to
$\mathrm{Re}(s)>r$. In that half-plane, only the term $i=r$ in
\eqref{eq:shifted-parity-series} has a pole; its residue is
$q_r\zeta(r+1)$ by~\eqref{eq:d1-parity-factor}.
The unshifted term in~\eqref{eq:unshifted-parity-series} has residue
$-q_r\zeta(r+1)$. Their difference therefore has residue
$2q_r\zeta(r+1)=R_r$, and the finite Dirichlet polynomial $E_f$ is entire.

\medskip

For the growth assertion, write $\alpha=r+1$ and
$E_r(x)=B_r(x)-R_rx^\alpha/\alpha$. Proposition~\ref{summatorygap}
gives $E_r(x)=O_r(x^r)$. Partial summation, initially for
$\mathrm{Re}(s)>\alpha$, yields
\begin{equation}\label{eq:gap-dirichlet-continuation}
\mathcal D_f(\Delta;s)
=\frac{R_rs}{\alpha(s-\alpha)}+s\int_1^\infty E_r(x)x^{-s-1}\;dx.
\end{equation}
The integral continues holomorphically to $\mathrm{Re}(s)>r$.
For $\mathrm{Re}(s)\ge r+\eta$ with $\eta>0$, it is bounded by
$C_r|s|/(\mathrm{Re}(s)-r)\le C_r|s|/\eta$.
This proves the asserted polynomial growth.
\end{proof}

\begin{lemma}\label{kernelestimate}
Let $h\in C^1([0,\infty))$ and suppose that there are constants
$C_h,c_h>0$ such that
\begin{align*}
|h(x)|+|h'(x)|\le C_he^{-c_hx}\qquad(x\ge1).    
\end{align*}
For each fixed non-negative integer $k$, as $t\downarrow0$,
\begin{equation}\label{eq:gap-kernel}
\sum_{j\ge1}j^kb_jh(tj)
=R_rt^{-r-k-1}\int_0^\infty x^{r+k}h(x)\;dx
+O_{r,h,k}(t^{-r-k}).
\end{equation}
The estimate is uniform for families of kernels with common bounds
for $h,h'$ on $[0,1]$ and common constants in the exponential decay
condition.
\end{lemma}

\begin{proof}
First take $k=0$. Extend $B_r(x)=\sum_{j\le x}b_j$ to $x>0$; it
vanishes when $0<x<1$. With
$E_r(x)=B_r(x)-R_rx^{r+1}/(r+1)$, we therefore have
\begin{align*}
|E_r(x)|\le C_rx^r\qquad(x>0),    
\end{align*}
after increasing $C_r$. Stieltjes integration by parts gives
\begin{align*}
\sum_{j\ge1}b_jh(tj)
&=-t\int_0^\infty B_r(x)h'(tx)\;dx\\
&=-\frac{R_rt}{r+1}\int_0^\infty x^{r+1}h'(tx)\;dx
-t\int_0^\infty E_r(x)h'(tx)\;dx.
\end{align*}
All boundary terms vanish: $B_r$ vanishes near zero, $h$ is bounded
there, and the decay assumption applies at infinity. Changing variables and integrating the first integral by parts give
\begin{align*}
-\frac{R_r}{r+1}t^{-r-1}
\int_0^\infty y^{r+1}h'(y)\;dy
=R_rt^{-r-1}\int_0^\infty y^rh(y)\;dy.    
\end{align*}
The second integral has absolute value at most
\begin{align*}
C_rt\int_0^\infty x^r|h'(tx)|\;dx
=C_rt^{-r}\int_0^\infty y^r|h'(y)|\;dy.    
\end{align*}
This proves the case $k=0$. For general $k$, the function
$H(x)=x^kh(x)$ and its derivative satisfy the same type of decay
condition. Apply the case already proved to $H$, and multiply the result by $t^{-k}$.
\end{proof}

\section{The Integral Asymptotics}\label{sec4}

As outlined in Section~\ref{initialsetup}, we have two main steps to establish the central limit theorem. First, to employ the saddle-point method, we analyse the asymptotic behaviour of the integrands

\begin{align}\label{maintermintegral}
\Omega_{1,f}(n)
\coloneqq\frac{e^{n\tau}}{2\pi}\int_{|\theta|\leq \theta_n}\exp(in\theta+F(\tau+i\theta,u))d\theta,
\end{align}
and
\begin{align}\label{errortermintegral}
\Omega_{2,f}(n)
\coloneqq\frac{e^{n\tau}}{2\pi}\int_{\theta_n<|\theta|\leq \pi}\exp(in\theta+F(\tau+i\theta,u))d\theta,
\end{align}
with a suitable choice of $\tau=\tau(n,u)$ such that the first-order derivative in~\eqref{taylorapproximation} vanishes.

We take $\tau$ to be the positive solution of $n=-F_\gamma(\tau,u)$ and set
\begin{equation}\label{arcwidth}
\theta_n=\tau^{1+3(r+1)/7}.
\end{equation}
The existence and uniqueness of the saddle will follow from the estimates below. All estimates are uniform for $\delta\le u\le\delta^{-1}$, where $0<\delta<1$ is fixed, and are understood as $n\to\infty$, or equivalently $\tau\downarrow0$. Their implied constants may depend on $r$ and $\delta$.

\subsection{Estimate of \texorpdfstring{$\Omega_{1,f}$}{Omega 1}}\label{EstimateofOmega1f}

From the definition of $F(\gamma,u)$, it is evident that deriving an asymptotic expression for the {integral} $\Omega_{1,f}$ necessitates a comprehensive understanding of the analytic properties of the Dirichlet series with shifted coefficient {$A_{\sigma_r}^{\varepsilon}(s)$} and {$D_{\sigma_r}^{\varepsilon}(s)$} corresponding to $f(n)=\sigma_r(n)$. {Their combination, with the finite correction in~\eqref{shiteddirichletseries}, is the absolute-gap Dirichlet series $\mathcal D_f(\Delta;s)$.}

\begin{lemma}\label{maintermforsigmar}
Let $f$ denote the $r$-th divisor function $\sigma_r(n)$ for fixed $r\geq 2$, and let $B^2=F_{\gamma\gamma}(\tau,u)$. We then have
\begin{align*}
\Omega_{1,f}
=\frac{e^{n\tau+F(\tau, u)}}{2\pi}\left(\int_{-\theta_n}^{\theta_n}\exp\left(-\frac{B^2}{2}\theta^2\right)d\theta\right)\left(1+O\left(\tau^{{2(r+1)/7}}\right)\right),
\end{align*}
where
\begin{align*}
\int_{-\theta_n}^{\theta_n}\exp\left(-\frac{B^2}{2}\theta^2\right)d\theta
=\frac{\sqrt{2\pi}}{B}+O\left(\tau^{{-1-3(r+1)/7}}B^{-2}\exp\left({-c_r\tau^{-(r+1)/7}}\right)\right)
\end{align*}
uniformly in $u${, for some constant $c_r>0$}.
\end{lemma}

\begin{proof}
Note that
\begin{align}\label{saddlepointequationforsigmak}
F(\tau+i\theta,u)
=F(\tau,u)+F_{\gamma}(\tau, u)i\theta-\frac{F_{\gamma\gamma}(\tau,u)}{2}\theta^2+O\bigg(\theta^3{\sup_{|\theta_0|\leq|\theta|}}|F_{\gamma\gamma\gamma}(\tau+i\theta_0,u)|\bigg).
\end{align}

\noindent
Utilising the definition of $\Delta_f(n)$ (as given in~\eqref{gfunctiondefintion}), we subsequently deduce that
\begin{align}\label{curlyfgammau}
F(\gamma,u)
=\sum_{n=1}^{\infty}\Delta_f(n)\log(1+ue^{-n\gamma}).
\end{align}
Here $\Delta_f(n)=b_n=|\sigma_r(n+1)-\sigma_r(n)|$. We first take
$\gamma=\tau>0$. Near the positive real axis, the logarithms are
continued from their real values; the estimates below justify the
complex derivatives on the central arc. The remaining arc will be
estimated directly through $Q_f$.

Choose $c>r+1$. Mellin inversion gives
\begin{align}\label{mellininversionforF}
F(\tau,u)
=\frac{1}{2\pi i}\int_{c-i\infty}^{c+i\infty}
 \mathcal D_f(\Delta;s)\mathcal{I}(s,u)\tau^{-s}\;ds .
\end{align}
Indeed, Mellin inversion for $\log(1+ue^{-x})$ is valid by
\eqref{mellinkerneldecay}, and we may sum its inverse transforms termwise:
$\sum_{n\geq1}b_nn^{-c}<\infty$ and
$\int_{\mathbb R}|\mathcal{I}(c+it,u)|\;dt<\infty$.
The same rapid decay permits any fixed number of differentiations
with respect to $\tau$ under the integral. Writing
$(s)_j=s(s+1)\cdots(s+j-1)$ for $j\geq1$, with $(s)_0=1$, we obtain
\begin{align}\label{mellininversionforderivatives}
\frac{\partial^j}{\partial\gamma^j}F(\tau,u)
=\frac{(-1)^j}{2\pi i}\int_{c-i\infty}^{c+i\infty}
 (s)_j\mathcal D_f(\Delta;s)\mathcal{I}(s,u)\tau^{-s-j}\;ds .
\end{align}
These formulas are uniform for $\delta\leq u\leq\delta^{-1}$.

\medskip

We now apply the converse mapping to the product in
\eqref{mellininversionforderivatives}. By Proposition~\ref{gapdirichlet},
its only pole between $\mathrm{Re}(s)=r+\tfrac12$ and
$\mathrm{Re}(s)=c$ is the simple pole at $s=r+1$, of residue
\begin{align*}
R_r(r+1)_j\;\mathcal{I}(r+1,u),
\qquad R_r=2^{-r}\zeta(r+1).
\end{align*}
There is no singularity of $\mathcal{I}(s,u)$ in this strip.
For completeness, shift the contour in
\eqref{mellininversionforderivatives} to
$\mathrm{Re}(s)=r+\tfrac12$. The Dirichlet factor grows at most
polynomially there, whereas~\eqref{mellinkerneldecay} gives arbitrarily large inverse powers of $|\mathrm{Im}(s)|$. Thus the integrals over the horizontal sides of the contour tend to zero. The integral over the new vertical line is absolutely convergent and is
$O(\tau^{-r-j-1/2})$, uniformly in $u$. The residue theorem consequently gives
\begin{align*}
\frac{\partial^j}{\partial\gamma^j}F(\tau,u)
=(-1)^jR_r(r+1)_j\;\mathcal{I}(r+1,u)\tau^{-r-j-1}
+O(\tau^{-r-j-1/2}).
\end{align*}
By~\eqref{mellinkernelidentity},
$\mathcal{I}(r+1,u)=-\Gamma(r+1)\mathrm{Li}_{r+2}(-u)$.

\medskip

The sharper remainder used below follows directly from the summatory
estimate in Section~3. Put $H_u(x)=\log(1+ue^{-x})$ and apply
Lemma~\ref{kernelestimate} with $k=j$ and $h=H_u^{(j)}$. Its hypotheses hold uniformly in $u$, because these kernels and their first derivatives are bounded at zero and decrease exponentially at infinity. Hence
\begin{align*}
\frac{\partial^j}{\partial\gamma^j}F(\tau,u)
=R_r\tau^{-r-j-1}
\int_0^\infty x^{r+j}H_u^{(j)}(x)\;dx+O(\tau^{-r-j}).
\end{align*}
Integration by parts $j$ times gives
\begin{align*}
\int_0^\infty x^{r+j}H_u^{(j)}(x)\;dx
=(-1)^j(r+1)_j\;\mathcal{I}(r+1,u);
\end{align*}
all boundary terms vanish by the same bounds on the kernels.
We therefore obtain
\begin{align}\label{fgammauderivativefinal}
\frac{\partial^j}{\partial\gamma^j}F({\tau},u)
={(-1)^{j+1}R_r\Gamma(r+j+1)\mathrm{Li}_{r+2}(-u)
 \tau^{-r-j-1}+O(\tau^{-r-j})}.
\end{align}

\noindent
Thus, for the saddle-point {$\tau=\tau(n,u)$}, we have
\begin{align}\label{saddlepointsolutionforsigmak}
n=-F_{\gamma}(\tau,u)
={-R_r\Gamma(r+2)\mathrm{Li}_{r+2}(-u)\tau^{-r-2}
 +O(\tau^{-r-1})\asymp\tau^{-r-2}},
\end{align}
whereas
\begin{align}\label{curvaturebound}
B^2=F_{\gamma\gamma}(\tau,u)
={-R_r\Gamma(r+3)\mathrm{Li}_{r+2}(-u)\tau^{-r-3}
 +O(\tau^{-r-2})\asymp\tau^{-r-3}}.
\end{align}

The constants in the last two leading terms are positive, since
$\mathcal{I}(r+1,u)>0$ by its integral representation. They are bounded
above and below by positive constants on the fixed compact interval
of $u$. The exact identities
\begin{align*}
-F_\gamma(\tau,u)=\sum_{k\geq1}\frac{u k b_k}{e^{k\tau}+u},
\qquad
F_{\gamma\gamma}(\tau,u)
=\sum_{k\geq1}\frac{u k^2b_ke^{k\tau}}{(e^{k\tau}+u)^2}>0
\end{align*}
show that $-F_\gamma(\tau,u)$ is continuous and strictly decreasing
from infinity to zero as $\tau$ increases from zero to infinity.
The positive saddle therefore exists uniquely and is a smooth function
of $u$ by the implicit function theorem.
Equation~\eqref{saddlepointsolutionforsigmak} also gives
$\tau\asymp n^{-1/(r+2)}$, uniformly in $u$.

\noindent
For the third derivative appearing in the error term of~\eqref{saddlepointequationforsigmak}, we have
\begin{align*}
\frac{\partial^3}{\partial \theta^3}F(\tau+i\theta,u)
=-i\sum_{k=1}^{\infty}\frac{k^3\Delta_f(k){u^{-1}e^{k(\tau+i\theta)}}(1-u^{-1}e^{k(\tau+i\theta)})}{(1+u^{-1}e^{k(\tau+i\theta)})^3}.
\end{align*}

\noindent
The following estimation technique is adapted from~\cite{madwag}. Let $k_0=\tau^{-(1+c)}$ for some {$0<c<3(r+1)/7$}, and for simplicity of notation, set $v=u^{-1}$. We consider two cases: first, when $k$ is large i.e., $k>k_0$; and second, when $k\leq k_0$. Beginning with the large sum, we obtain
\begin{align}\label{thirdorderderivative}
&\left|-i\sum_{k>k_0}\frac{k^3\Delta_f(k){ve^{k(\tau+i\theta)}}(1-ve^{k(\tau+i\theta)})}{(1+ve^{k(\tau+i\theta)})^3}\right|
\leq\sum_{k>k_0}\frac{k^3|\Delta_f(k)|{ve^{k\tau}}(1+ve^{k\tau})}{|1+ve^{k(\tau+i\theta)}|^3}\nonumber\\
\leq&\sum_{k>k_0}\frac{k^3|\Delta_f(k)|{ve^{k\tau}}(1+ve^{k\tau})}{(ve^{k\tau}-1)^3}
=O\bigg(\sum_{k>k_0}\frac{k^3|\Delta_f(k)|}{e^{k\tau}}\bigg)
=O\left(\tau^{-r-4}\right),
\end{align}
where we have applied the upper bound (see, for instance,~\cite{ramanujan})
\begin{align*}
\sigma_r(n)
<n^r(1+2^{-r}+3^{-r}+\cdots)
{=}n^r\zeta(r).
\end{align*}
{In particular, $|\Delta_f(k)|\ll_r k^r$, as also proved in
Lemma~\ref{gapbounds}. The denominators in~\eqref{thirdorderderivative}
are positive for all sufficiently small $\tau$, since
$ve^{k\tau}\ge\delta e^{\tau^{-c}}\to\infty$ when $k>k_0$.}

\noindent
For the remaining case, we observe that
\begin{align*}
\left|1+{ve^{k(\tau+i\theta)}}\right|
\geq (1+ve^{k\tau})\cos{\left(\frac{k\theta}{2}\right)}.
\end{align*}
{For $k\le k_0$ and $|\theta|\le\theta_n$,
$|k\theta|\le\tau^{3(r+1)/7-c}=o(1)$.
The cosine is therefore positive and bounded away from zero.
The displayed inequality follows by squaring both sides and using
$4ve^{k\tau}\le(1+ve^{k\tau})^2$.}

\noindent
Thus, we obtain
\begin{align*}
&\left|-i\sum_{k\leq k_0}\frac{k^3\Delta_f(k){u^{-1}e^{k(\tau+i\theta)}}(1-u^{-1}e^{k(\tau+i\theta)})}{(1+u^{-1}e^{k(\tau+i\theta)})^3}\right|
\leq\sum_{k\leq k_0}\frac{k^3|\Delta_f(k)|{ve^{k\tau}}(1+ve^{k\tau})}{|1+ve^{k(\tau+i\theta)}|^3}\\
\leq&\sum_{k\leq k_0}{\frac{k^3|\Delta_f(k)|ve^{k\tau}}{(1+ve^{k\tau})^2}}\left(1+O((k\theta)^2)\right)
\leq \sum_{k\leq k_0}\frac{k^3|\Delta_f(k)|}{ve^{k\tau}}\left(1+O((k\theta)^2)\right)\\
\leq&\sum_{k\geq 1}\frac{k^3|\Delta_f(k)|}{ve^{k\tau}}+O\bigg(\sum_{k\geq 1}\frac{k^5|\Delta_f(k)|\theta^2}{e^{k\tau}}\bigg).
\end{align*}

\noindent
{Applying $|\Delta_f(k)|\ll_r k^r$ and integral comparison, we have}
\begin{align*}
\left|-i\sum_{{k\le k_0}}\frac{k^3\Delta_f(k){u^{-1}e^{k(\tau+i\theta)}}(1-u^{-1}e^{k(\tau+i\theta)})}{(1+u^{-1}e^{k(\tau+i\theta)})^3}\right|
=O\left(\tau^{-r-4}+\tau^{-r-6}\theta^2\right).
\end{align*}

\noindent
Combining this with~\eqref{thirdorderderivative}, we get
\begin{align*}
\frac{\partial^3}{\partial\theta^3}F(\tau+i\theta,u)
=O(\tau^{-r-4}),
\end{align*}
for $|\theta|\leq\tau^{{1+3(r+1)/7}}$. All these together leads to the expansion
\begin{align*}
F(\tau+i\theta,u)
=F(\tau,u){-in\theta}-\frac{B^2}{2}\theta^2+O(\tau^{-r-4}\theta^3).
\end{align*}

\noindent
Thus for the integral $\Omega_{1,f}$, we obtain
\begin{align*}
&\frac{e^{n\tau}}{2\pi}\int_{-\theta_n}^{\theta_n}\exp(in\theta+F(\tau+i\theta,u))d\theta\\
=&\frac{e^{n\tau+F(\tau,u)}}{2\pi}\int_{-\theta_n}^{\theta_n}\exp\left(-\frac{B^2}{2}\theta^2+O(\tau^{{2(r+1)/7}})\right)d\theta\\
=&\frac{e^{n\tau+F(\tau,u)}}{2\pi}\left(\int_{-\theta_n}^{\theta_n}\exp\left(-\frac{B^2}{2}\theta^2\right)d\theta\right)\left(1+O(\tau^{{2(r+1)/7}})\right).
\end{align*}
{Here $\tau^{-r-4}\theta_n^3=\tau^{2(r+1)/7}\to0$, so the
last expansion is uniform over the integration interval.}

\noindent
Finally, we transform the integral into a Gaussian integral and arrive at
\begin{align*}
\int_{-\theta_n}^{\theta_n}\exp\left(-\frac{B^2}{2}\theta^2\right)d\theta
=&\int_{-\infty}^{\infty}\exp\left(-\frac{B^2}{2}\theta^2\right)d\theta
-2\int_{\theta_n}^{\infty}\exp\left(-\frac{B^2}{2}\theta^2\right)d\theta\\
=&\frac{\sqrt{2\pi}}{B}+O\left(\int_{\theta_n}^{\infty}\exp\left(-\frac{B^2\tau^{{1+3(r+1)/7}}}{2}\theta\right)d\theta\right)\\
=&\frac{\sqrt{2\pi}}{B}+O\left(\tau^{{-1-3(r+1)/7}}B^{-2}\exp\left({-c_r\tau^{-(r+1)/7}}\right)\right),
\end{align*}
{where we used $B^2\theta_n^2\asymp\tau^{-(r+1)/7}$ from
\eqref{curvaturebound}.} {This completes} the proof.
\end{proof}

\subsection{Estimate of \texorpdfstring{$\Omega_{2,f}$}{Omega 2}}

We proceed to prove the following lemma, which establishes an upper bound for the integrand in~\eqref{errortermintegral} when $f(n)=\sigma_r(n)$ for fixed $r\geq 2$.

\begin{lemma}\label{omega2sigmak}
For a fixed $r\geq 2$ with $f(n)=\sigma_r(n)$, the error term is bounded as follows
\begin{align*}
|\Omega_{2,\sigma_r}|
=O\left(\exp\left(n\tau+F(\tau,u)-c_{r}\tau^{{-(r+1)/7}}\right)\right),
\end{align*}
where $c_{r}$ is a {positive} constant uniformly in $u$.
\end{lemma}

{Lemma~\ref{gapbounds} supplies a positive lower bound of order
$n^r$ for the absolute gap when $n$ is sufficiently large.}
Additionally, we incorporate the following lemma by Li and Chen~\cite[Lemma 2.5]{lichen} as a key element of the {following} argument.

\begin{lemma}[Li-Chen,~\cite{lichen}]\label{lichenlemma}
Let $k\geq 1$ be {a fixed integer}. Then for $\xi>0$, we have
\begin{align*}
\sum_{n\geq 1}n^{k-1}e^{-n\xi}(1-\cos(ny))\geq \varrho\bigg(\frac{e^{-\xi}}{(1-e^{-\xi})^{k}}-\frac{e^{-\xi}}{\left|1-e^{-\xi-iy}\right|^{k}}\bigg),
\end{align*}
where $\varrho$ is a positive constant which depends only on $k$.
\end{lemma}

\begin{proof}
One may take $\varrho=1$. Put $z=e^{-\xi}$. For integers $n,k\ge1$,
\begin{align*}
\binom{n+k-2}{k-1}\le n^{k-1}.
\end{align*}
For $k\ge2$ this follows by comparing each factor
$(n+h)/(h+1)$ with $n$, for $0\le h\le k-2$; the case $k=1$ is
immediate. Since $1-\cos(ny)\ge0$ and
$\sum_{n\ge1}\binom{n+k-2}{k-1}z^n=z/(1-z)^k$, we have
\begin{align*}
\sum_{n\ge1}n^{k-1}z^n(1-\cos(ny))
&\ge\frac{z}{(1-z)^k}
 -\Re\left(\frac{ze^{iy}}{(1-ze^{iy})^k}\right)\\
&\ge\frac{z}{(1-z)^k}-\frac{z}{|1-ze^{iy}|^k}.
\end{align*}
The last modulus equals $|1-ze^{-iy}|$, which proves the claim.
\end{proof}

We are now ready to prove the following result, which will directly imply Lemma~\ref{omega2sigmak}. {We apply the preceding lemma with $k=r+1$.}

\begin{lemma}\label{qlemmaforomega2sigmak}
For a fixed $r\geq 2$ and any real $y$ satisfying ${\theta_n}\leq |y|\leq \pi$, we have {for all sufficiently small $\tau$}
\begin{align*}
\frac{|Q_{\sigma_r}({e^{-\tau-iy}},u)|}{{Q_{\sigma_r}(e^{-\tau},u)}}
\leq\exp\left(-\frac{2u{c_r}}{(1+u)^2}{\tau^{-(r+1)/7}}\right),
\end{align*}
where {$c_r>0$ depends only on $r$}.
\end{lemma}

\begin{proof}
By Lemma~\ref{gapbounds}, there are constants $g_r>0$ and $J_r$
such that $\Delta_f(n)\ge g_r n^r$ for $n\ge J_r$.
The finitely many omitted indices contribute a uniformly bounded
amount, because $0\le1-\cos(ny)\le2$. Thus, for a constant $C_r>0$,
\begin{align}\label{lowerboundforgn}
\sum_{n\ge1}\Delta_f(n)e^{-n\tau}(1-\cos(ny))
\ge g_r\sum_{n\ge1}n^re^{-n\tau}(1-\cos(ny))-C_r.
\end{align}

\noindent
Next, we express
\begin{align*}
\left|1+ue^{-n(\tau+iy)}\right|^2
=(1+ue^{-n\tau-niy})(1{+}ue^{-n\tau+niy})
=(1+ue^{{-n\tau}})^2-2ue^{-n\tau}(1-\cos(ny)).
\end{align*}
{Indices with $\Delta_f(n)=0$ are omitted from the product and
logarithmic sums below. If a factor with positive weight vanishes,
the desired inequality is immediate; otherwise all these logarithms
are well defined.}
\noindent
Using the above identities, we obtain
\begin{align*}
\left(\frac{|Q_{\sigma_r}({e^{-\tau-iy}},u)|}{{Q_{\sigma_r}(e^{-\tau},u)}}\right)^2
=&\prod_{n=1}^{\infty}\bigg(1-\frac{2ue^{-n\tau}(1-\cos(ny))}{(1+ue^{-n\tau})^2}\bigg)^{\Delta_f(n)}\\
\leq &\exp\bigg(\sum_{n=1}^{\infty}|\Delta_f(n)|\log\left|1-\frac{2ue^{-n\tau}(1-\cos(ny))}{(1+ue^{-n\tau})^2}\right|\bigg)\\
\leq&\exp\bigg(-\sum_{n=1}^{\infty}|\Delta_f(n)|\left|\frac{2ue^{-n\tau}(1-\cos(ny))}{(1+ue^{-n\tau})^2}\right|\bigg)\\
\leq&\exp\bigg(-\frac{2u}{(1+u)^2}\sum_{n=1}^{\infty}|\Delta_f(n)|e^{-n\tau}(1-\cos(ny))\bigg)\\
\leq&\exp\bigg(-\frac{2u{g_r}}{(1+u)^2}\sum_{n=1}^{\infty}n^{{r}}e^{-n\tau}(1-\cos(ny)){+\frac{2uC_r}{(1+u)^2}}\bigg).
\end{align*}
{The weights $\Delta_f(n)$ are nonnegative and
$\log(1-x)\le-x$ applies to the factors above.}

\noindent
Finally, applying Lemma~\ref{lichenlemma}, {with $k=r+1$,} yields
\begin{align}\label{mainexpression}
\left(\frac{|Q_{\sigma_r}({e^{-\tau-iy}},u)|}{{Q_{\sigma_r}(e^{-\tau},u)}}\right)^2
\leq\exp\left(-\frac{2u{g_r}}{(1+u)^2}\varrho\left(\frac{e^{-\tau}}{(1-e^{-\tau})^{{r+1}}}-\frac{e^{-\tau}}{\left|1-e^{-\tau-iy}\right|^{{r+1}}}\right){+\frac{2uC_r}{(1+u)^2}}\right).
\end{align}

\medskip
\noindent
Beginning with a few standard trigonometric identities from~\cite{lichen}, we express
\begin{align}\label{trig1}
\left|1-e^{-\tau-i\theta}\right|^{{r+1}}
=&((1-e^{-\tau})^2+2e^{-\tau}(1-\cos\theta))^{{(r+1)/2}}\nonumber\\
\geq & (1-e^{-\tau})^{{r+1}}\left(1+\frac{2e^{-\tau}}{(1-e^{-\tau})^2}(1-\cos\theta_n)\right)^{{(r+1)/2}},
\end{align}
where $\theta_n\leq |\theta|\leq \pi$. By employing the Taylor series expansion of $\cos{\theta}$, we obtain
\begin{align*}
1-\cos\theta_n
=\frac{1}{2}\theta_n^2+O(\theta_n^4)
=\frac{1}{2}\tau^{{2+6(r+1)/7}}+O({\tau^{4+12(r+1)/7}}),
\end{align*}
and similarly, $e^{-\tau}=1-\tau+O(\tau^2)$. Consequently, we derive
\begin{align*}
\frac{2e^{-\tau}}{(1-e^{-\tau})^2}(1-\cos\theta_n)
=\frac{2(1+O(\tau))(\tfrac{1}{2}\tau^{{2+6(r+1)/7}}+O({\tau^{4+12(r+1)/7}}))}{\tau^2(1+O(\tau))}
=\tau^{{6(r+1)/7}}(1+O(\tau)).
\end{align*}

\medskip
\noindent
Substituting the above expression into~\eqref{trig1} gives
\begin{align*}
\left|1-e^{-\tau-i\theta}\right|^{{r+1}}
\geq& (1-e^{-\tau})^{{r+1}}\left(1+\tau^{{6(r+1)/7}}(1+O(\tau))\right)^{{(r+1)/2}}\\
=&(1-e^{-\tau})^{{r+1}}\left(1+\frac{{(r+1)}\tau^{{6(r+1)/7}}}{2}
+O\left(\tau^{{6(r+1)/7+1}}+\tau^{{12(r+1)/7}}\right)\right).
\end{align*}

\medskip
\noindent
Therefore, the inner {difference} of~\eqref{mainexpression} {satisfies}
\begin{align*}
&\frac{e^{-\tau}}{(1-e^{-\tau})^{{r+1}}}-\frac{e^{-\tau}}{\left|1-e^{-\tau-iy}\right|^{{r+1}}}\\
&\geq\frac{e^{-\tau}}{(1-e^{-\tau})^{{r+1}}}
\bigg(1-\frac{1}{1+\frac{{(r+1)}\tau^{{6(r+1)/7}}}{2}
+O\left(\tau^{{6(r+1)/7+1}}+\tau^{{12(r+1)/7}}\right)}\bigg)\\
&=\frac{e^{-\tau}}{(1-e^{-\tau})^{{r+1}}}
\left(\frac{{(r+1)}\tau^{{6(r+1)/7}}}{2}
+O\left(\tau^{{6(r+1)/7+1}}+\tau^{{12(r+1)/7}}\right)\right)\\
&=\frac{{1+O(\tau)}}{\tau^{{r+1}}}
\left(\frac{{(r+1)}\tau^{{6(r+1)/7}}}{2}
+O\left(\tau^{{6(r+1)/7+1}}+\tau^{{12(r+1)/7}}\right)\right)\\
&=\frac{{r+1}}{2}\tau^{{-(r+1)/7}}
+O\left(\tau^{{1-(r+1)/7}}+\tau^{{5(r+1)/7}}\right)
\gg\tau^{{-(r+1)/7}}.
\end{align*}

\noindent
Substituting this bound into~\eqref{mainexpression}, {the quantity
inside the negative exponential tends to infinity uniformly in $y$.
The finite correction $C_r$ can therefore be absorbed after decreasing
the positive constant. Taking the square root then} {concludes} the proof.
\end{proof}

\medskip
\noindent
{\it Proof of Lemma~\ref{omega2sigmak}.} By the definition given in~\eqref{errortermintegral} for $f(n)=\sigma_r(n)$, the following {inequality} holds
\begin{align*}
|\Omega_{2,\sigma_r}|
\leq\frac{e^{n\tau}}{{\pi}}\int_{\theta_n}^{\pi}\left|{Q}_{\sigma_r}(e^{-\tau-i\theta},u)\right|d\theta
=\frac{e^{n\tau+F(\tau,u)}}{\pi}\int_{\theta_n}^{\pi}\frac{|Q_{\sigma_r}(e^{-\tau-i\theta},u)|}{Q_{\sigma_r}(e^{-\tau},u)}d\theta.
\end{align*}
{Here the moduli on the positive and negative arcs agree because
the weights and $u$ are real.}

\noindent
By employing Lemma~\ref{qlemmaforomega2sigmak}, we arrive at
\begin{align*}
|\Omega_{2,\sigma_r}|
\ll\exp\left(n\tau+F(\tau,u)-c_{r}\tau^{{-(r+1)/7}}\right),
\end{align*}
where $c_{r}$ is a constant that remains uniform in $u$, as noted.
\qed

\section{The Moment Generating Function}\label{sec5}

In this section, we employ the estimates of $\Omega_{1,f}$ and $\Omega_{2,f}$ to derive
\begin{align*}
\mathcal{Q}_{f,n}(u)
=\frac{e^{n\tau+F(\tau,u)}}{{\sqrt{2\pi}\;B}}\left(1+O\left({\tau^{2(r+1)/7}}\right)\right),
\end{align*}
where $r\geq 2$ is fixed and $B^2=F_{\gamma\gamma}(\tau,u)$.

We adopt the method outlined in~\cite[\S4.3]{lipniketal}, by considering the asymptotic expansion of the moment generating function associated with the random variable ${\varpi}_{f,n}$. Let $M_{f,n}(\theta)=\mathbb{E}(e^{(\varpi_{f,n}-\mu_{f,n})\theta/\nu_{f,n}})$ denote the moment generating function for real $\theta$, and let $\mu_{f,n}$ and $\nu_{f,n}$ represent the {saddle centring and scaling parameters}, as defined in~\eqref{meansigma} and~\eqref{variancesigma}, respectively. This leads to the following result.

{Write $\kappa=(r+1)/(r+2)$. We shall use $\nu_{f,n}^2\asymp_r n^\kappa$, which is proved directly from Lemma~\ref{kernelestimate} in the following subsection, independently of the moment generating function.}

\begin{lemma}\label{momentgeneratingfunctionlemma}
Let $r\geq 2$ be fixed and $f(n)=\sigma_r(n)$ be the generalised divisor function. Then as $n\to\infty$, for bounded $\theta$, the following holds
\begin{align*}
M_{f,n}(\theta)
=\exp\left(\frac{\theta^2}{2}+O\left({n^{-2\kappa/7}}\right)\right).
\end{align*}
More precisely, there is a fixed $c>0$ such that, uniformly for real $|\theta|\le c\nu_{f,n}$,
\begin{align*}
\log M_{f,n}(\theta)
=\frac{\theta^2}{2}
+O_r\left(n^{-\kappa/2}(|\theta|+|\theta|^3)+n^{-2\kappa/7}\right).
\end{align*}
\end{lemma}

\begin{proof}
We begin by noting that
\begin{align}\label{asymtoticsofmomentgeneratingfunction}
M_{f,n}(\theta)
=&\exp\left(-\frac{\mu_{f,n}\theta}{\nu_{f,n}}\right)\frac{\mathcal{Q}_{f,n}(e^{\theta/\nu_{f,n}})}{\mathcal{Q}_{f,n}(1)}\nonumber\\
=&\sqrt{\frac{B^2(\tau(n,1),1)}{B^2(\tau(n,e^{\theta/\nu_{f,n}}),e^{\theta/\nu_{f,n}})}}\exp\bigg(-\frac{\mu_{f,n}\theta}{\nu_{f,n}}+n\tau(n,e^{\theta/\nu_{f,n}})\nonumber\\
&\quad\quad\quad+{F}(\tau(n,e^{\theta/\nu_{f,n}}),e^{\theta/\nu_{f,n}})-n\tau(n,1)-{F}(\tau(n,1),1)+{O_r(n^{-2\kappa/7})}\bigg).
\end{align}
{Here the relative error in the coefficient estimate is uniformly $o(1)$, so its logarithm has the same order.}
\noindent
For simplicity of notation, we define $\tau=\tau(n,u)$ which is a function of $n$ and $u$ with respect to $u$. By employing the saddle-point solution given in~\eqref{saddlepointsolutionforsigmak}, we obtain (following the approach of~\cite{madwag})

\begin{align*}
\tau_u
=\tau_{u}(n,u)
=-\frac{{F}_{u\gamma}(\tau,u)}{{F}_{\gamma\gamma}(\tau,u)}
=\frac{\sum_{k=1}^{\infty}\frac{k\Delta_f(k)e^{k\tau}}{{(e^{k\tau}+u)^2}}}{\sum_{k=1}^{\infty}\frac{uk^2\Delta_f(k){e^{k\tau}}}{{(e^{k\tau}+u)^2}}},
\end{align*}
and similarly,
\begin{align*}
\tau_{uu}=\tau_{uu}(n,u)
=\frac{1}{{F}_{\gamma\gamma}(\tau,u)^3}\big(-{F}_{\gamma\gamma\gamma}(\tau,u){F}_{u\gamma}(\tau,u)^2
&+2{F}_{u\gamma\gamma}(\tau,u){F}_{u\gamma}(\tau,u){F}_{\gamma\gamma}(\tau,u)\\
&-{F}_{uu\gamma}(\tau,u){F}_{\gamma\gamma}(\tau,u)^2\big),
\end{align*}
and
\begin{align*}
\tau_{uuu}=\tau_{uuu}(n,u)
=&\frac{1}{{F}_{\gamma\gamma}(\tau,u)^5}\big(-{F}_{uuu\gamma}(\tau,u){F}_{\gamma\gamma}(\tau,u)^4\\
&+(3{F}_{uu\gamma\gamma}(\tau,u){F}_{u\gamma}(\tau,u)+3{F}_{uu\gamma}(\tau,u){F}_{u\gamma\gamma}(\tau,u)){F}_{\gamma\gamma}(\tau,u)^3\\
&+\big(-3{F}_{u\gamma\gamma\gamma}(\tau,u){F}_{u\gamma}(\tau,u)^2-6{F}_{u\gamma\gamma}(\tau,u)^2{F}_{u\gamma}(\tau,u)\\
&-3{F}_{uu\gamma}(\tau,u){F}_{\gamma\gamma\gamma}(\tau,u){F}_{u\gamma}(\tau,u)\big)F_{\gamma\gamma}(\tau,u)^2\\
&+\big({F}_{\gamma\gamma\gamma\gamma}(\tau,u){F}_{u\gamma}(\tau,u)^3+9{F}_{u\gamma\gamma}(\tau,u){F}_{\gamma\gamma\gamma}(\tau,u){F}_{u\gamma}(\tau,u)^2\big){F}_{\gamma\gamma}(\tau,u)\\
&-3{F}_{u\gamma}(\tau,u)^3{F}_{\gamma\gamma\gamma}(\tau,u)^2\big).
\end{align*}

\noindent
As it is evident from the preceding expression, it is necessary to estimate the partial derivative of ${F}_{\gamma}(\tau,u)$ with respect to $u$.

Termwise differentiation is justified on compact subsets of $\tau>0$, $u>0$ by exponential convergence. Since $\Delta_f(k)\ll_r k^r$, for nonnegative integers $a,b$ with $a+b\le4$ we have
\begin{align*}
\left|\frac{\partial^{a+b}}{\partial\gamma^a\partial u^b}F(\tau,u)\right|
\ll_r \sum_{k\ge1}k^{r+a}e^{-k\tau}
\ll_r\tau^{-r-a-1}.
\end{align*}
These estimates are uniform for $u$ in the fixed compact positive interval; the constants may also depend on that interval. Since $\Delta_f(k)\gg_r k^r$ for all sufficiently large $k$, restricting the positive sums for $-F_\gamma$ and $F_{\gamma\gamma}$ to $\tau^{-1}\le k\le2\tau^{-1}$ also gives
\begin{align*}
n=-F_\gamma(\tau,u)\asymp_r\tau^{-r-2},
\qquad
F_{\gamma\gamma}(\tau,u)\asymp_r\tau^{-r-3}.
\end{align*}
In particular, $\tau\asymp_r n^{-1/(r+2)}$. For the first mixed derivative, we obtain
\begin{align*}
F_{u\gamma}(\tau,u)
=-\sum_{k\ge1}\frac{k\Delta_f(k)e^{k\tau}}{(e^{k\tau}+u)^2}
=O_r(\tau^{-r-2})=O_r(n).
\end{align*}
In a similar fashion, we derive the remaining partial derivatives and get
\begin{align*}
&F_{u}(\tau,u)
=\sum_{k=1}^{\infty}\frac{\Delta_f(k)}{e^{k\tau}+u}
={O_r(\tau^{-(r+1)})}
=O(n^{\frac{r+1}{r+2}}),\\
&F_{uu}(\tau,u)
=-\sum_{k=1}^{\infty}\frac{\Delta_f(k)}{(e^{k\tau}+u)^2}
={O_r(\tau^{-(r+1)})}
=O(n^{\frac{r+1}{r+2}}),\\
&F_{\gamma\gamma}(\tau,u)
=\sum_{k=1}^{\infty}\frac{uk^2\Delta_f(k)e^{k\tau}}{(e^{k\tau}+u)^2}
={O_r(\tau^{-(r+3)})}
=O(n^{\frac{r+3}{r+2}}),\\
&F_{uuu}(\tau,u)
=\sum_{k=1}^{\infty}\frac{2\Delta_f(k)}{(e^{k\tau}+u)^3}
=O(\tau^{-(r+1)})
=O(n^{\frac{r+1}{r+2}}),\\
&F_{uu\gamma}(\tau,u)
=\sum_{k=1}^{\infty}\frac{2k\Delta_f(k)e^{k\tau}}{(e^{k\tau}+u)^3}
=O(\tau^{-(r+2)})
=O(n),\\
&F_{u\gamma\gamma}(\tau,u)
=\sum_{k=1}^{\infty}\frac{k^2\Delta_f(k)e^{k\tau}(e^{k\tau}-u)}{(e^{k\tau}+u)^3}
=O(\tau^{-(r+3)})
=O(n^{\frac{r+3}{r+2}}),\\
&F_{\gamma\gamma\gamma}(\tau,u)
=-\sum_{k=1}^{\infty}\frac{uk^3\Delta_f(k)e^{k\tau}(e^{k\tau}-u)}{(e^{k\tau}+u)^3}
=O(\tau^{-(r+4)})
=O(n^{\frac{r+4}{r+2}}),\\
&F_{uuu\gamma}(\tau,u)
=-\sum_{k=1}^{\infty}\frac{6k\Delta_f(k)e^{k\tau}}{(e^{k\tau}+u)^4}
=O(\tau^{-(r+2)})
=O(n),\\
&F_{uu\gamma\gamma}(\tau,u)
=-\sum_{k=1}^{\infty}\frac{2k^2\Delta_f(k)e^{k\tau}({2e^{k\tau}}-u)}{(e^{k\tau}+u)^4}
=O(\tau^{-(r+3)})
=O(n^{\frac{r+3}{r+2}}),\\
&F_{u\gamma\gamma\gamma}(\tau,u)
=-\sum_{k=1}^{\infty}\frac{k^3\Delta_f(k)e^{k\tau}(e^{2k\tau}-4ue^{k\tau}+u^2)}{(e^{k\tau}+u)^4}
=O(\tau^{-(r+4)})
=O(n^{\frac{r+4}{r+2}}),\\
&F_{\gamma\gamma\gamma\gamma}(\tau,u)
=\sum_{k=1}^{\infty}\frac{uk^4\Delta_f(k)e^{k\tau}(e^{2k\tau}-4ue^{k\tau}+u^2)}{(e^{k\tau}+u)^4}
=O(\tau^{-(r+5)})
=O(n^{\frac{r+5}{r+2}}).
\end{align*}

\noindent
The above bounds imply that the derivatives $\tau_u$, $\tau_{uu}$, and $\tau_{uuu}$ are each $O(n^{-1/(r+2)})$, uniformly with respect to $u$. {Choose $c>0$ so that $e^{[-c,c]}$ lies in the interior of the fixed interval for $u$. The Taylor remainders below are then uniform for $|\theta|\le c\nu_{f,n}$.} Expanding $\tau(n, e^{\theta/\nu_{f,n}})$ and ${F(\tau(n,e^{\theta/\nu_{f,n}}), e^{\theta/\nu_{f,n}})}$ around $\theta=0$ yields
\begin{align*}
\tau(n,e^{\theta/\nu_{f,n}})-\tau(n,1)
=\tau_{u}(n,1)\frac{\theta}{\nu_{f,n}}+\frac{\tau_{u}(n,1)+\tau_{uu}(n,1)}{2}\left(\frac{\theta}{\nu_{f,n}}\right)^2+O\bigg(n^{-1/(r+2)}\frac{|\theta|^3}{\nu^3_{f,n}}\bigg)
\end{align*}
and
\begin{align*}
&{F(\tau(n,e^{\theta/\nu_{f,n}}), e^{\theta/\nu_{f,n}})-F(\tau(n,1),1)}\\
=&{\big(F_{\gamma}(\tau(n,1),1)\tau_{u}(n,1)+F_{u}(\tau(n,1),1)\big)\frac{\theta}{\nu_{f,n}}}
+\frac{1}{2}\bigg(F_{\gamma}(\tau(n,1),1)(\tau_{u}(n,1)+\tau_{uu}(n,1))\\
&+F_{\gamma\gamma}(\tau(n,1),1)\tau_{u}(n,1)^2+2F_{u\gamma}(\tau(n,1),1)\tau_{u}(n,1)+F_{u}(\tau(n,1),1)\\
&+F_{uu}(\tau(n,1),1)\bigg)\left(\frac{\theta}{\nu_{f,n}}\right)^2+O\bigg(n^{\frac{r+1}{r+2}}\frac{|\theta|^3}{\nu^3_{f,n}}\bigg),
\end{align*}
respectively. Substituting these expansions into the exponential in~\eqref{asymtoticsofmomentgeneratingfunction}, we obtain
\begin{align*}
&\left(n\tau_u(n,1)+F_{\gamma}(\eta,1)\tau_u(n,1)+F_u(\eta,1)-\mu_{f,n}\right)\frac{\theta}{\nu_{f,n}}\\
&+\frac{1}{2}\bigg(n(\tau_u(n,1)+\tau_{uu}(n,1))+F_{\gamma}(\eta,1)(\tau_u(n,1)\\
&+\tau_{uu}(n,1))
+F_{\gamma\gamma}(\eta,1)\tau_u(n,1)^2+2F_{u\gamma}(\eta,1)\tau_u(n,1)+F_{u}(\eta,1)+F_{uu}(\eta,1)\bigg)\left(\frac{\theta}{\nu_{f,n}}\right)^2\\
&+O\bigg(n^{\frac{r+1}{r+2}}\frac{|\theta^3|}{\nu^3_{f,n}}+{n^{-2\kappa/7}}\bigg).
\end{align*}
Here, we have written $\eta = \tau(n, 1)$ for brevity. Recalling that $n = -F_{\gamma}(\eta, 1)$ and $\tau_u(n, 1) = -\dfrac{F_{u\gamma}(\eta, 1)}{F_{\gamma\gamma}(\eta, 1)}$, we may simplify the final expression to derive
\begin{align*}
&(F_{u}(\eta,1)-\mu_{f,n})\frac{\theta}{\nu_{f,n}}+\frac{1}{2}\bigg(F_{u}(\eta,1)+F_{uu}(\eta,1)-\frac{F_{u\gamma}(\eta,1)^2}{F_{\gamma\gamma}(\eta,1)}\bigg)\left(\frac{\theta}{\nu_{f,n}}\right)^2\\
&+O\bigg(n^{\frac{r+1}{r+2}}\frac{|\theta^3|}{\nu^3_{f,n}}+{n^{-2\kappa/7}}\bigg).
\end{align*}

\noindent
In a similar manner, we obtain that
\begin{align*}
\frac{B^2(\tau(n,1),1)}{B^2(\tau(n,e^{\theta/\nu_{f,n}}),e^{\theta/\nu_{f,n}})}
=1+O\bigg(\frac{|\theta|}{\nu_{f,n}}\bigg).
\end{align*}
Indeed,
\begin{align*}
\frac{d}{du}\log B^2(\tau(n,u),u)
=\frac{F_{u\gamma\gamma}(\tau,u)+F_{\gamma\gamma\gamma}(\tau,u)\tau_u}
{F_{\gamma\gamma}(\tau,u)}
=O_r(1).
\end{align*}
Integrating between $u=1$ and $u=e^{\theta/\nu_{f,n}}$ proves the displayed estimate.

\noindent
Thus, we derive the following asymptotic expression for the moment generating function in~\eqref{asymtoticsofmomentgeneratingfunction}:
\begin{align*}
M_{f,n}(\theta)
=&\exp\bigg((F_{u}(\eta,1)-\mu_{f,n})\frac{\theta}{\nu_{f,n}}+\frac{1}{2}\bigg(F_{u}(\eta,1)+F_{uu}(\eta,1)-\frac{F_{u\gamma}(\eta,1)^2}{F_{\gamma\gamma}(\eta,1)}\bigg)\left(\frac{\theta}{\nu_{f,n}}\right)^2\\
&+O\bigg({\frac{|\theta|}{\nu_{f,n}}+}n^{\frac{r+1}{r+2}}\frac{|\theta^3|}{\nu^3_{f,n}}+{n^{-2\kappa/7}}\bigg)\bigg).
\end{align*}

\noindent
Recall that we choose $\mu_{f,n}$ and $\nu_{f,n}$ in~\eqref{meansigma} and~\eqref{variancesigma} such that
\begin{align}\label{meanvarianceformulatosolve}
&\mu_{f,n}
=F_{u}(\eta,1)
=\sum_{k=1}^{\infty}\frac{\Delta_f(k)}{e^{\eta k}+1}\quad\text{and}\quad\nonumber\\
&\nu^2_{f,n}
=F_{u}(\eta,1)+F_{uu}(\eta,1)-\frac{F_{u\gamma}(\eta,1)^2}{F_{\gamma\gamma}(\eta,1)}
=\sum_{k=1}^{\infty}\frac{\Delta_f(k)e^{\eta k}}{(e^{\eta k}+1)^2}-\frac{\left(\sum_{k=1}^{\infty}\frac{\Delta_f(k)ke^{\eta k}}{(e^{\eta k}+1)^2}\right)^2}{\sum_{k=1}^{\infty}\frac{\Delta_f(k)k^2e^{\eta k}}{(e^{\eta k}+1)^2}}.
\end{align}

\noindent
From the definitions of $\mu_{f,n}$ and $\nu^2_{f,n}$ in~\eqref{meansigma} and~\eqref{variancesigma} respectively, we deduce that
\begin{align*}
M_{f,n}(\theta)
=&\exp\bigg(\frac{\theta^2}{2}+O\left({n^{-\kappa/2}}|\theta|+{n^{-\kappa/2}}|\theta|^3+{n^{-2\kappa/7}}\right)\bigg)\\
=&\exp\bigg(\frac{\theta^2}{2}+O\left({n^{-2\kappa/7}}\right)\bigg),
\end{align*}
for bounded $\theta$.
\end{proof}

\subsection{Formulae for the mean and variance} 

By applying Lemma~\ref{momentgeneratingfunctionlemma} and Curtiss' theorem~\cite{curtiss}, we establish that the distribution of ${\varpi}_{f,n}$ is indeed asymptotically normal. {Since the limiting distribution function is continuous, the convergence of the distribution functions is uniform in $x$.} In this section, we derive {asymptotic formulae for the saddle centring and variance scale}, providing an explicit formula for the constant as a function of $r$.

To relate the saddle quantities to the actual moments, put
$Z_n=(\varpi_{f,n}-\mu_{f,n})/\nu_{f,n}$. For every fixed $s>0$, the preceding lemma gives
\begin{align*}
\mathbb E(e^{s|Z_n|})
\le M_{f,n}(s)+M_{f,n}(-s)=O_{r,s}(1).
\end{align*}
For $j=1,2$,
\begin{align*}
\mathbb E\bigl(|Z_n|^j\boldsymbol{1}_{\{|Z_n|>H\}}\bigr)
\le \sup_{x>H}(x^je^{-sx})\;\mathbb E(e^{s|Z_n|})
\longrightarrow0
\end{align*}
uniformly for large $n$ as $H\to\infty$. Thus the first two moments converge to the corresponding normal moments:
$\mathbb E Z_n\to0$ and $\mathbb E Z_n^2\to1$. It follows that
\begin{align*}
\mathbb E\varpi_{f,n}=\mu_{f,n}+o(\nu_{f,n}),
\qquad
\operatorname{Var}(\varpi_{f,n})\sim\nu_{f,n}^2.
\end{align*}
Put
\begin{align*}
R_r&=2^{-r}\zeta(r+1),\\
A_r&=R_r\Gamma(r+2)(1-2^{-r-1})\zeta(r+2),\\
M_r&=R_r\Gamma(r+1)(1-2^{-r})\zeta(r+1),\\
V_r&=R_r\Gamma(r+1)\left((1-2^{1-r})\zeta(r)
-\frac{r+1}{r+2}
 \frac{((1-2^{-r})\zeta(r+1))^2}
 {(1-2^{-r-1})\zeta(r+2)}\right).
\end{align*}
For $a>0$, expanding $(e^x+1)^{-1}$ into exponentials gives
\begin{align*}
\int_0^\infty\frac{x^a}{e^x+1}\;dx
=\Gamma(a+1)(1-2^{-a})\zeta(a+1).
\end{align*}
The interchange is justified by
$\sum_{\ell\ge1}\int_0^\infty x^ae^{-\ell x}\;dx
=\Gamma(a+1)\zeta(a+1)<\infty$.
With $h(x)=e^x/(e^x+1)^2$, integration by parts therefore gives
\begin{align*}
\int_0^\infty x^ah(x)\;dx
=\Gamma(a+1)(1-2^{1-a})\zeta(a)
\qquad(a>1).
\end{align*}
Lemma~\ref{kernelestimate} now yields, as $\eta\downarrow0$,
\begin{align*}
\sum_{k\ge1}\frac{k\Delta_f(k)}{e^{\eta k}+1}
=A_r\eta^{-r-2}+O_r(\eta^{-r-1})
\end{align*}
and, for $j=0,1,2$,
\begin{align*}
\sum_{k\ge1}k^j\Delta_f(k)h(\eta k)
=R_r\Gamma(r+j+1)(1-2^{1-r-j})\zeta(r+j)\eta^{-r-j-1}
+O_r(\eta^{-r-j}).
\end{align*}
The sum with $j=2$ has a positive leading coefficient, so these estimates may be substituted in the quotient in~\eqref{meanvarianceformulatosolve}.

From~\eqref{meanvarianceformulatosolve} we write
\begin{align*}
&\mu_{f,n}
=F_{u}(\eta,1)
={M_r\eta^{-(r+1)}+O_r(\eta^{-r})}
\end{align*}
and 
\begin{align*}
\nu^2_{f,n}
=&F_{u}(\eta,1)+F_{uu}(\eta,1)-\frac{F_{u\gamma}(\eta,1)^2}{F_{\gamma\gamma}(\eta,1)}\\
=&{V_r\eta^{-(r+1)}+O_r(\eta^{-r})}.
\end{align*}

\noindent
Utilising the saddle-point solution in~\eqref{saddlepointsolutionforsigmak}, we write
\begin{align*}
\frac{1}{\eta}
=\frac{1}{\tau(n,1)}
={(n/A_r)^{1/(r+2)}+O_r(1)}
\end{align*}
then 
\begin{align*}
\mu_{f,n}
={C_{\mu}(r)n^\kappa+O_r(n^{r/(r+2)})}
\end{align*}
and
\begin{align*}
\nu^2_{f,n}
={C_{\sigma}(r)n^\kappa+O_r(n^{r/(r+2)})},
\end{align*}
where 
\begin{align}\label{meanconstant}
C_{\mu}(r)
{=M_rA_r^{-\kappa}}
\end{align}
and
\begin{align}\label{varianceconstant}
C_{\sigma}(r)
{=V_rA_r^{-\kappa}}.
\end{align}
To verify positivity, write
\begin{align*}
I_j=\int_0^\infty x^{r+j}h(x)\;dx\qquad(j=0,1,2).
\end{align*}
The strict Cauchy-Schwarz inequality for the measure $x^rh(x)\;dx$ gives $I_1^2<I_0I_2$, since the functions $1$ and $x$ are not proportional almost everywhere. Using
\begin{align*}
\frac{\Gamma(r+2)^2}{\Gamma(r+3)}
=\frac{r+1}{r+2}\Gamma(r+1),
\end{align*}
we obtain
\begin{align*}
V_r=R_r\left(I_0-\frac{I_1^2}{I_2}\right)>0.
\end{align*}
The corresponding discrete Cauchy-Schwarz inequality, applied to the weights $\Delta_f(k)h(\eta k)$ in~\eqref{meanvarianceformulatosolve}, is also strict because these weights are positive at two distinct indices. Thus $\nu_{f,n}^2>0$, and the asymptotic formula proves $\nu_{f,n}^2\asymp_r n^\kappa$, as used above.

\subsubsection{The distribution of the tail}

We conclude the proof by bounding the tail of the distribution, following the approach of Lipnik et al.~\cite{lipniketal}, which in turn was adapted from Hwang~\cite{hwang}. For $\theta=o(n^{\frac{r+1}{6(r+2)}})$, {with $\theta\ge0$,} we apply the Chernoff bound (see, for instance,~\cite{chernoff}) to obtain
\begin{align*}
\mathbb{P}\left(\frac{\varpi_{f,n}-\mu_{f,n}}{{\nu_{f,n}}}\geq x\right)
\leq e^{-\theta x} M_{f,n}(\theta)
=e^{{-x\theta}+\theta^2/2}\left(1+O\left({n^{-\kappa/2}}(|\theta|+|\theta|^3)+{n^{-2\kappa/7}}\right)\right).
\end{align*}

\noindent
Setting $T=n^{\frac{r+1}{6(r+2)}}/\log n$ we have two cases. {Uniformly for $|\theta|\le T$, the error in the last display is $O_r((\log n)^{-3})$, since $n^{-\kappa/2}T^3=(\log n)^{-3}$ and $n^{-2\kappa/7}=o((\log n)^{-3})$.} For the first case, let ${0\leq}x\leq T$, and we set $\theta=x$, yielding
\begin{align*}
\mathbb{P}\left(\frac{\varpi_{f,n}-\mu_{f,n}}{{\nu_{f,n}}}\geq x\right)  
\leq e^{-x^2/2}\left(1+O((\log n)^{-3})\right).
\end{align*}

\noindent
Similarly, for $x\geq T$, we set $\theta=T$, obtaining 
\begin{align*}
\mathbb{P}\left(\frac{\varpi_{f,n}-\mu_{f,n}}{{\nu_{f,n}}}\geq x\right)  
\leq e^{-Tx/2}\left(1+O((\log n)^{-3})\right).
\end{align*}

\noindent
A similar argument can be applied to estimate the probability $\mathbb{P}(\frac{\varpi_{f,n}-\mu_{f,n}}{{\nu_{f,n}}}\leq -x)$.

\medskip
\noindent
{\it{Acknowledgement}.} We are grateful to Kyle Pratt and Alexandru Zaharescu for their valuable suggestions and many insightful discussions. A special thanks to Oleksiy Klurman for his comments on an earlier version of this paper. {We also thank the referees for their careful reading and helpful suggestions.} The first author is supported by EPSRC doctoral fellowship \#EP/T518049/1.

\medskip
\noindent
{\it{Rights Retention}}. For the purpose of open access, the author has applied a Creative Commons Attribution (CC BY) licence to any Author Accepted Manuscript version arising from this submission.

\printbibliography

\end{document}